\documentclass[11pt, twoside]{article}

\usepackage{amssymb}
\usepackage{amsmath}
\usepackage{amsthm}
\usepackage{color}
\usepackage{xcolor}
\usepackage{mathrsfs}
\usepackage[active]{srcltx}
\usepackage{esint}
\allowdisplaybreaks

\def\rn{{\mathbb{R}^n}}

\def\fz{\infty}

\def\dz{\delta}
\def\bdz{\Delta}
\def\ez{\epsilon}

\def\kz{\kappa}
\def\lz{\lambda}

\def\oz{\omega}

\def\sz{\sigma}

\def\lf{\left}
\def\r{\right}

\def\hs{\hspace{0.25cm}}
\def\ls{\lesssim}
\def\gs{\gtrsim}

\def\dist{\mathop\mathrm{\,dist\,}}
\def\loc{{\mathop\mathrm{\,loc\,}}}
\def\supp{\mathop\mathrm{\,supp\,}}

\newtheorem{theorem}{Theorem}[section]
\newtheorem{lemma}[theorem]{Lemma}

\newtheorem{proposition}[theorem]{Proposition}

\theoremstyle{definition}
\newtheorem{remark}[theorem]{Remark}
\newtheorem{definition}[theorem]{Definition}
\renewcommand{\appendix}{\par
	\setcounter{section}{0}%
	\setcounter{subsection}{0}%
	\setcounter{subsubsection}{0}%
	\gdef\thesection{\@Alph\c@section}%
	\gdef\thesubsection{\@Alph\c@section.\@arabic\c@subsection}%
	\gdef\theHsection{\@Alph\c@section.}%
	\gdef\theHsubsection{\@Alph\c@section.\@arabic\c@subsection}%
	\csname appendixmore\endcsname
}

\numberwithin{equation}{section}

\begin{document}
	
\arraycolsep=1pt
	
\title{\bf\Large  Predual Spaces of Hardy Spaces Related to
Fractional Schr\"odinger Operators
\footnotetext{\hspace{-0.2cm}2020
{\it Mathematics Subject Classification}. {Primary 42B35; Secondary 42B30, 30H35.}
\endgraf{\it Key words and phrases.} Fractional Schr\"odinger operator, Hardy space, VMO space, predual space,
tent space}}
\author{Qiumeng Li, Haibo Lin\footnote{Corresponding author}\ and Sibei Yang}
\date{}
	
\maketitle
	
\vspace{-0.8cm}
	
\begin{center}
\begin{minipage}{13cm}
{\small {\bf Abstract.}\quad Let $n\in\mathbb{N}$ and ${\alpha}\in(0,\min\{2,n\})$. For any
$a\in[a^\ast,\infty)$, the fractional Schr\"odinger operator $L_\alpha$ is defined by
\begin{equation*}
L_\alpha:=(-\Delta)^{{\alpha}/2}+a{|x|}^{-{\alpha}},
\end{equation*}
where $a^*:=-{\frac{2^{\alpha}{\Gamma}((d+{\alpha})/4)^2}{{\Gamma}((d-{\alpha})/4)^2}}$. Let
$\gamma\in[0,\frac{\alpha}{n})$. In this paper, we introduce the VMO-type spaces $\mathrm{VMO}_{L_\alpha}^\gamma
(\mathbb{R}^{n})$ associated with $L_\alpha$, and characterize these spaces via some tent spaces. We also
prove that, for any given $p\in(\frac{n}{n+\alpha},1]$, the space $\mathrm{VMO}_{L_\alpha}^{\frac{1}{p}-1}
(\mathbb{R}^{n})$ is the predual space of the Hardy space $H_{L_\alpha}^p\lf(\mathbb{R}^{n}\r)$ related to $L_\alpha$. }
\end{minipage}
\end{center}
	
\section{Introduction}\label{s1}
\hskip\parindent
The theory of real variables in function spaces is one of the central contents of harmonic analysis. The
classical function spaces studied in harmonic analysis include Lebesgue spaces, BMO spaces, Hardy spaces, etc.
The BMO space is the bounded mean oscillation space, which is a kind of function space introduced by John and
Nirenberg \cite{jn61} when they studied the solutions of elliptic partial differential equations. It contains
the $L^\infty$ space and is the dual space of the Hardy space \cite{g141,g142}. It is well known that BMO spaces
have extensive applications in complex analysis, probability theory, and partial differential equations. In
the 1970s, Coifman and Weiss \cite{cw77} introduced a function space of vanishing mean oscillation, which is
recorded as VMO$(\mathbb{R}^{n})$. The VMO space is defined as the closure of continuous function space with
compact support under BMO norm. The VMO space is a subspace of the BMO space, which is widely concerned in PDEs.
Then Coifman and Weiss \cite{cw77} proved that the Hardy space is the dual space of the VMO space. In 1978,
Uchiyama \cite{u78} conducted a more in-depth study on VMO$(\mathbb{R}^{n})$, and he proved that VMO$(\mathbb{R}^{n})$
spaces can be described by the average oscillation limit on the cube. It is worth noting that a characteristic
of VMO spaces is also given by the compactness of the commutator of singular integrals. Classical BMO and VMO
theories are closely related to the Laplace operator $\bdz$. The generalization of operators also brings new
challenges to the study of BMO spaces and VMO spaces. In 1926, after the famous Austrian physicist Schr\"odinger
further promoted the concept of de Broglie matter wave, he obtained the equation of motion that the matter
wave satisfies in the quantum system, namely, the Schr\"odinger equation
$$ih\frac{\partial \psi}{\partial t}=\lf(-\frac{h^2}{2m}\bdz+V\r)\psi.$$
For the study of this equation, ones can solve relevant problems by studying the following Schr\"odinger
operator $L:=-\bdz+V$. Since it was proposed, it has been concerned by many scholars and has made fruitful
research results; see, for instance, \cite{s82,blm04,d95,rr06} and their references. In particular, Dziuba\'nskiet al.
\cite{dgmtz05} first studied the BMO space related to Schr\"odinger operators whose potential satisfies
the reverse H\"older inequality. See \cite{ttd21,yyz10} for recent progress.
	
The heat kernel of Schr\"odinger operators whose potential satisfies the reverse H\"older inequality have
the Poisson upper bound estimate. Recall that the kernel is said to satisfy the Poisson upper bound estimate of order
$m\in(0,\fz)$, if it satisfies
\begin{equation*}
|p_t(x,y)|\le h_t(x,y):=t^{-\frac{n}{m}}s\lf(\frac{|x-y|}{t^{1/m}}\r),\ \forall\ x,y\in\mathbb{R}^{n},
\end{equation*}
where $s$ is a positive bounded increasing function satisfying that, for every $\ez\in(0,\infty)$,
$\lim_{r\rightarrow\infty}r^{n+\ez}s(r)=0$. Let $L$ be the operator that the kernel of $e^{-tL}$
satisfies the Poisson upper bound estimation. In 2005, Duong et al. \cite{dy051} introduced the Hardy space
$H_L^1(\mathbb{R}^{n})$ via the area square function, and then established its molecular and atomic
characterizations, and proved that its dual space is the BMO space.
In 2008, Deng et al. \cite{ddsty08} introduced and characterized a new function space VMO$_L(\mathbb{R}^{n})$
of vanishing mean oscillation related to the operator $L$, and
proved that it is the predual space of $H_L^1(\mathbb{R}^{n})$. In addition, Yan \cite{y08} generalized the
results in \cite{ddsty08} to $H_L^p(\mathbb{R}^{n})$ with some $p\in(0,1]$. In 2011, Jiang and Yang
\cite{jy112} introduced a VMO type space related to the conjugate operator $L^*$ of $L$, and proved that it is
the predual space of the Banach completion space $B_{\oz,L}(\mathbb{R}^{n})$ of Orlicz--Hardy space
$H_{\oz,L}(\mathbb{R}^{n})$.
	
The Schr\"odinger operator with magnetic fields and real potentials satisfies the Davies--Gaffney estimate.
Let $(\mathrm{X},d,\mu)$ be a space of homogeneous type in the sense of Coifman and Weiss \cite{cw71,cw77}, which includes the Euclidean space with Lebesgue measures as a special case. For the operator family $\{T_t\}$, if there are
positive constants $C$ and $c$ such that, for any $t\in(0,\fz)$,
\begin{align*}
|\langle T_tf_1,f_2\rangle|\le C\exp\lf\{-\frac{[\dist(U_1,U_2)]^2}{ct}\r\}
\|f_1\|_{L^2{(\mathrm{X})}}\|f_2\|_{L^2{(\mathrm{X})}}
\end{align*}
holds for any $f_i\in L^2{(\mathrm{X})}$, $\supp f_i\subset U_i$, $U_i\subset \mathrm{X}$
with $i\in\{1,2\}$, then it is said that $\{T_t\}_{t>0}$ satisfies the Davies--Gaffney estimate. Suppose that
the semigroup generated by the operator $L$ satisfies the Davies--Gaffney estimate. In 2011, Hofmann et al.
developed the theory of Hardy spaces and BMO spaces related to operator $L$ in \cite{hlmmy11}, including
atomic (or molecular) characterization, area function characterization, and duality of Hardy spaces. In 2010,
Jiang and Yang \cite{jy10} studied the theory of VMO-type spaces related to the operator $L$.

To our best knowledge, most of the existing harmonic analysis theories related to operators are based on the fact that the
heat kernel satisfies the Poisson upper bound estimate or the operator satisfies the Davies--Gaffney estimate.
For more research on function spaces related to operators; see, for instance,
\cite{adm05,amr08,ar03,cy22,ddsty08,ddy05,dl13,dxy07,dy051,dy052,hlmmy11,hm09,jy10,jy111,jy112,y04,y08} and their references.

In 2019, Bogdan et al. \cite{bgjp19} established the heat kernel estimation of semigroup generated by
the fractional Schr\"odinger operator $L_\alpha:=(-\Delta)^{\alpha/2}+\kz|x|^\alpha$, $\kz\in(0,\infty)$.
See \cite{jxhs10,zy09} for relevant research on the heat kernel estimation of this operator.
We consider the following fractional Schr\"odinger operators on $\mathbb{R}^{n}$, $n\in\mathbb{N}$,
\begin{equation}\label{1.1}
L_{\alpha}:=(-\Delta)^{\alpha / 2}+a|x|^{-\alpha} \quad \text {with} \quad \alpha\in(0,\min \{2, n\})
\quad \text { and } \quad a\in [a^{*},\infty),
\end{equation}
where	
$$
a^{*}:=-\frac{2^{\alpha} \Gamma((d+\alpha) / 4)^{2}}{\Gamma((d-\alpha) / 4)^{2}}.
$$
The constant $a^{*}$ is reasoned by the sharp constant in the following Hardy-type inequality
$$
\int_{\mathbb{R}^{n}}|x|^{-\alpha}|u(x)|^{2}\,dx\leq-\frac{1}{a^{*}}
\int_{\mathbb{R}^{n}}|\xi|^{\alpha}|\widehat{u}(\xi)|^{2} d \xi, \quad
u \in C_{0}^{\infty}(\mathbb{R}^{n}),
$$
where $\widehat{u}$ denotes the Fourier transform of $u$. Hence, the condition $a\in [a^{*},\infty)$
guarantees that the operator $L_{\alpha}$ is non-negative and self-adjoint; see, for instance, \cite{bb21,fms21,h77}.
The operator $L_{\alpha}$ can be viewed as a Sch\"odinger operator of the fractional Laplacian $L_{\alpha}
=(-\Delta)^{\alpha/2}+V$ with $V:=a|x|^{-\alpha}$. In this case the potential $V$ might be negative.
The heat kernel of the operator $L_\alpha$ does not satisfy the Poisson upper bound estimate and, moreover
$L_\alpha$ also does not satisfy the Davies--Gaffney estimate; see, for instance, \cite[Theorems 2.1 and 2.2]{bd21}.
In 2022, Bui and Nader \cite{bn22} introduced the Hardy space $H_{L_\alpha}^p(\mathbb{R}^{n})$ via the area square function, and then established its molecular characterization,
and proved that its dual space is the BMO space.

Inspired by the above research, in this paper, we introduce the VMO-type space $\mathrm{VMO}_{L_\alpha,M}^\gamma
(\mathbb{R}^{n})$ associated with the fractional Schr\"odinger operator $L_\alpha$, and then prove that it is the predual space
of the Hardy space $H_{L_\alpha}^p(\mathbb{R}^{n})$. Although the approach in our paper bases on those in
\cite{jy112,bn22}, a number of important improvements and modifications are needed. This is because the heat
kernel of the operator $L_\alpha$ does not satisfy the Poisson upper bound estimate and the operator does not
satisfy the Davies--Gaffney estimate.

The organization of this paper is as follows. In Section 2, we first recall the kernel estimate of the operator
$L_\alpha$ and some related properties, then we recall the definitions of Hardy spaces and BMO spaces related to the operator
$L_\alpha$, and establish some necessary lemmas related to the operator $L_\alpha$, such as the Carleson
measure characterization of BMO spaces. Section 3 is devoted to define VMO-type spaces $\mathrm{VMO}_{L_\alpha,M}^\gamma
(\mathbb{R}^{n})$ and establish the Carleson measure characterization of these spaces. In Section 4,
we give the proof of
$$\lf(\mathrm{VMO}_{L_{\alpha}, M_0}^{\frac{1}{p}-1}(\mathbb{R}^{n})\r)^{*}=H_{L_\alpha}^p(\mathbb{R}^{n})$$
for $p\in(\frac{n}{n+\alpha},1]$.

Finally, we make some conventions on notations. We assume that $C$ is a positive constant, which is independent
of the main parameter, but may vary from line to line. Marked constants, such as $c_{(M,M_1)}$ and $C_{(\alpha,M)}$,
do not change in different cases. The \emph{symbol} $f\ls g$ means that $f\le Cg$. If $f\ls g$ and $g\ls f$,
we then write $f\sim g$. If $f\le Cg$ and $g=h$ or $g\le h$, we then write $f\ls g\sim h$
or $f\ls g\ls h$, \emph{rather than} $f\ls g=h$ or $f\ls g\le h$. For any ball $B\subseteq \rn$, we use $x_B$ and $r_B$ to denote its center and radius,
respectively. Given a ball $B$ and $\lz\in(0,\,\infty)$, $\lz B$ represents the ball which has the same center
of $B$ and the radius is $\lz$ times of $B$. For any $p\in(1,\,\infty)$, let $p':=p/(p-1)$ be its \emph{conjugate index}.
Given a measurable set $E\subset \rn$, we denote the characteristic function of $E$ with $\chi_E$. We use
$\mathbb{N}$ to denote the set of all positive integers, that is, $\mathbb{N}:=\{1,2,\dots\}$.
In addition, for any $f\in L^1(E)$, we denote the integral $\int_{E}|f(x)|\,dx$ simply by
$\int_{E}|f|\,dx$ and, when $|E|<\fz$, we use the notation
$$
\fint_Ef:=\frac{1}{|E|}\int_{E}f(x)\,dx.
$$

\section{Preliminaries}\label{s2}
	
To present our results, we first introduce some necessary notation and notions.
	
We set
$$
\Psi_{\alpha, n}(\delta):=-2^{\alpha} \frac{\Gamma\left(\frac{\delta+\alpha}{2}\right)
\Gamma\left(\frac{n-\delta}{2}\right)}{\Gamma\left(\frac{n-\delta-\alpha}{2}\right)
\Gamma\left(\frac{\delta}{2}\right)}, \quad \delta \in(-\alpha,(n-\alpha)/2] \backslash\{0\},
$$
and $\Psi_{\alpha, n}(0):=0$. The function $\Psi_{\alpha, n}$ is continuous and strictly decreasing in
$(-\alpha,(n-\alpha)/2]$ with
$$
\lim _{\delta \rightarrow-\alpha} \Psi_{\alpha, n}(\delta)=\infty \quad \text { and } \quad
\Psi_{\alpha, n}\left(\frac{n-\alpha}{2}\right)=a^{*}.
$$
For any $a\in[a^{*},\infty)$, we define
\begin{equation}\label{2.1}
\sigma:=\Psi_{\alpha, n}^{-1}(a),
\end{equation}
then $\sigma \in(-\alpha,(n-\alpha)/2]$.
	
For a constant $\sigma \in \mathbb{R}$, we denote
$$
D_{\sigma}(x,t):=\left(1+\frac{t^{\frac{1}{\alpha}}}{|x|}\right)^{\sigma}
$$
and
\begin{equation*}
n_\sz:=\lf\{\begin{array}{rcl}
\frac{n}{\sz} & & ,{\sz\in(0,\infty)},\\
\infty & & ,{\sz\in(-\infty,0]}.
\end{array}\r.
\end{equation*}

\begin{lemma}[{\cite[Lemma 2.3]{bn22}}]\label{l2.1}
Let $\sz\in(-\infty,n)$. Then there exist positive constants $C,\widetilde{C}$ such that, for any
$t\in(0,\,\infty)$, any $B(x,r_B)$ with $x\in\mathbb{R}^n$ and $r_B\in(0,\,\infty)$ and any $f\in L^2_{\loc}
(\mathbb{R}^n)$,
\begin{align*}
\int_{B(x, r_B)}D_\sz (z,t)dz\le C \min\lf\{t^\frac{n}{\alpha}+r_B^n,t^\frac{\sz}{\alpha}r_B^{n-\sz}+r_B^n\r\}
\end{align*}
and
\begin{align*}
\int_{B(x, r_B)}|D_\sz (z,t)f(z)|dz\le \widetilde{C} t^\frac{n}{2\alpha}\|f\|_{L^2(B)}+\|f\|_{L^1(B(x,r_B))}.
\end{align*}
\end{lemma}

In what follows, for a given ball $B$, the associated annuli are the sets $S_j(B):=(2^jB)\backslash(2^{j-1}B)$
for $j=1,2,3,\ldots,$ while we write $S_0(B):=B$. We now recall the following heat kernel estimate in
\cite{bd21,bn22}; see also \cite{bgjp19}.

\begin{theorem}[{\cite[Theorem 3.1]{bd21}}]\label{t2.2}
Let $\left\{T_{t}\right\}_{t>0}$ be a family of linear operators on $L^{2}(\mathbb{R}^{n})$ with their
associated kernels $\{T_{t}(\cdot,\cdot)\}_{t>0}$. Assume that there exist a positive constant $C$ and
$\sz\in(-\infty, n]$ such that, for all $t\in(0,\,\infty)$ and $x, y \in \mathbb{R}^{n} \backslash\{0\}$,
$$
\left|T_{t}(x, y)\right| \leq C t^{-n / \alpha}\left(\frac{t^{1 / \alpha}+|x-y|}{t^{1 / \alpha}}
\right)^{-n-\alpha} D_{\sz}(x, t) D_{\sz}(y, t) .
$$
Assume that $n_{\sz}^{\prime}<p \leq q<n_{\sz}$. Then there exists a positive constant $c$ such that, for any ball $B$, for every $t\in(0,\infty)$ and $j \in \mathbb{N}$, we have
$$
\begin{aligned}
\left(\fint_{S_{j}(B)}\left|T_{t} f\right|^{q}\right)^{1 / q} &\leq c \max \left\{\left(\frac{r_{B}}{t^{1 /
\alpha}}\right)^{n / p},\left(\frac{r_{B}}{t^{1 / \alpha}}\right)^{n}\right\}\left(1+\frac{t^{1 / \alpha}}{2^{j} r_{B}}\right)^{n / q}\\
&\hs\times\left(1+\frac{2^{j} r_{B}}{t^{1 / \alpha}}\right)^{-n-\alpha}\left(\fint_{B}|f|^{p}\right)^{1 / p}
\end{aligned}
$$
for any $f \in L^{p}(\mathbb{R}^{n})$ supported in $B$, and
\begin{align*}
\left(\fint_{B}\left|T_{t} f\right|^{q}\right)^{1 / q} &\leq c \max \left\{\left(\frac{2^{j} r_{B}}{t^{1 / \alpha}}\right)^{n},
\left(\frac{2^{j} r_{B}}{t^{1 / \alpha}}\right)^{n / p}\right\}\left(1+\frac{t^{1 / \alpha}}{r_{B}}\right)^{n / q}\\
&\hs\times\left(1+\frac{2^{j} r_{B}}{t^{1 / \alpha}}\right)^{-n-\alpha}\left(\fint_{S_{j}(B)}|f|^{p}\right)^{1 / p}
\end{align*}
for any $f \in L^{p}\left(S_{j}(B)\right)$.
\end{theorem}
	
\begin{proposition}[{\cite[Proposition 3.5]{bd21}}]\label{p2.3}
Let $n \in \mathbb{N}$, $\alpha\in(0,\min\{2, n\}),$ and let $\sigma$ be defined by \eqref{2.1}.
Then, for any $k \in \mathbb{N}$, there exists $C_{(k)}\in(0,\,\infty)$ such that, for all $t\in(0,\,\infty)$ and
$x, y \in \mathbb{R}^{n} \backslash\{0\}$,
$$
\left|p_{t, k}(x, y)\right| \leq C_{(k)} t^{-(k+n/ \alpha)} D_{\sigma}(x, t) D_{\sigma}(y, t)\left(\frac{t^{1 /
\alpha}+|x-y|}{t^{1 / \alpha}}\right)^{-n-\alpha},
$$
where $p_{t, k}(x, y)$ is an associated kernel to $\mathcal{L}_{\alpha}^{k} e^{-t \mathcal{L}_{\alpha}}$.
\end{proposition}

For $f \in L^{2}(\mathbb{R}^{n})$ and $x \in \mathbb{R}^{n}$, we define the area square function by
$$
S_{L_{\alpha}}(f)(x):=\left(\int_{0}^{\infty} \int_{B(x, t)}\left|t^{\alpha} L_{\alpha} e^{-t^{\alpha} L_{\alpha}}
f(y)\right|^{2}\,\frac{d y d t}{t^{n+1}}\right)^{\frac{1}{2}}.
$$

The Hardy space $H_{L_\alpha}^p(\mathbb{R}^{n})$ and BMO space associated with the operator $L_\alpha$ are
defined as follows; see, for instance, \cite{bn22}.

\begin{definition}[{\cite[Definition 3.1]{bn22}}]\label{d2.4}
For $p \in(0,1]$, we set
$$
\mathbb{H}_{S_{L_{\alpha}}}^{p}(\mathbb{R}^{n}):=\left\{f \in L^{2}(\mathbb{R}^{n}): \ S_{L_{\alpha}}(f)
\in L^{p}(\mathbb{R}^{n})\right\} .
$$
We then defined the Hardy space $H_{S_{L_{\alpha}}}^{p}(\mathbb{R}^{n})$ to be the completion of the set
$\mathbb{H}_{S_{L_{\alpha}}}^{p}(\mathbb{R}^{n})$ under the quasi-norm
$$\|f\|_{H_{S_{L_{\alpha}}}^{p}(\rn)}:=\left\|S_{L_{\alpha}}(f)\right\|_{p} .$$
\end{definition}

\begin{definition}[{\cite[Definition 3.2]{bn22}}]\label{d2.5}
For $p\in(0,1]$, $\epsilon\in(0,\infty)$ and $M\in\mathbb{N}$, we define $m\in L^{2}(\mathbb{R}^{n})$ to
be a $(p, 2, M, \epsilon)_{L_{\alpha}}$-molecule associated to $L_{\alpha}$, if there exist a function
$b \in D\left(L_\alpha^{M}\right)$ and a ball $B=B\left(x_{B}, r_{B}\right)$ such that
\begin{itemize}
\item [{\rm (i)}]$m=L_{\alpha}^{M} b$.
		
\item [{\rm (ii)}]For every $k=0,1, \ldots, M$ and $j \in \mathbb{N}$, one has
$$
\left\|\left(r_{B}^{\alpha} L_{\alpha}\right)^{k} b\right\|_{L^{2}\left(S_{j}(B)\right)}
\leq 2^{-j \epsilon} r_{B}^{\alpha M}\left|2^{j} B\right|^{\frac{1}{2}-\frac{1}{p}} .
$$
\end{itemize}
\end{definition}

\begin{definition}[{\cite[Definition 3.3]{bn22}}]\label{d2.6}
Given $p\in(0,1]$, $\epsilon\in(0,\infty)$ and $M\in\mathbb{N}$, we say that $f=\sum_{i} \lambda_{i} m_{i}$
is a molecular $(p, 2, M, \epsilon)_{L_{\alpha}}$-representation of $f$ if
$(\sum_{j=1}^{\infty}|\lambda_{i}|^{p})^{\frac{1}{p}}<\infty$, each $m_{j}$ is a $(p, 2, M, \epsilon)$-molecule,
and the sum converges in $L^{2}(\mathbb{R}^{n})$. Set
$$
\mathbb{H}_{L_{\alpha}, \text { mol }, M, \epsilon}^{p}(\mathbb{R}^{n}):=\left\{f \in L^{2}(\mathbb{R}^{n}):\
f \text { has a }(p, 2, M, \epsilon)_{L_{\alpha}} \text {-representation}\right\}
$$
for $p\in(0,1]$, with the quasi-norm given by
$$
\begin{aligned}
&\|f\|_{H_{L_{\alpha}, \mathrm{mol}, M, \epsilon}^{p}(\rn)} \\
&:=\inf \left\{\left(\sum_{i=1}^{\infty}\left|\lambda_{i}\right|^{p}\right)^{\frac{1}{p}}:\
f=\sum_{i=0}^{\infty} \lambda_{i} m_{i} \text { is a }(p, 2, M, \epsilon)_{L_{\alpha}}\text
{-representation}\right\} .
\end{aligned}
$$
For $p\in(0,1]$, the space $H_{L_{\alpha}, \text { mol }, M, \epsilon}^{p}(\mathbb{R}^{n})$ is defined to be
the completion of $\mathbb{H}_{L_{\alpha}, \mathrm{mol}, M, \epsilon}^{p}(\mathbb{R}^{n})$ with respect to
the quasi-norm $\|\cdot\|_{H_{L_{\alpha}, \mathrm{mol}, M, \epsilon}^{p}(\rn)}$.
\end{definition}

\begin{theorem}[{\cite[Theorem 3.4]{bn22}}]\label{t2.7}
Let $p\in(\frac{n}{n+\alpha},1]$, $\epsilon=\alpha+n-\frac{n}{p}$, and $M\in(\frac{n}{\alpha p},\infty)\cap\mathbb{N}$.
Then the Hardy spaces $H_{L_{\alpha}, \operatorname{mol}, M, \epsilon}^{p}(\mathbb{R}^{n})$ and
$H_{S_{L_{\alpha}}}^{p}(\mathbb{R}^{n})$ coincide with equivalent quasi-norms.
\end{theorem}

\begin{remark}[{\cite[Remark 3.9]{bn22}}]\label{r2.8}
Due to the coincidence between $H_{L_{\alpha}, \text { mol }, M, \epsilon}^{p}(\mathbb{R}^{n})$ and
$H_{S_{L_{\alpha}}}^{p}(\mathbb{R}^{n})$, in the sequel we will write $H_{L_{\alpha}}^{p}(\mathbb{R}^{n})$
for either $H_{L_{\alpha}, \text { mol }, M, \epsilon}^{p}(\mathbb{R}^{n})$ or
$H_{S_{L_{\alpha}}}^{p}(\mathbb{R}^{n})$ with $p\in(\frac{n}{n+\alpha},1]$, $\epsilon=\alpha+n-\frac{n}{p}$,
and $M\in(\frac{n}{\alpha p},\infty)\cap\mathbb{N}$.
\end{remark}

In \cite{bn22}, it is proved that the dual of the Hardy space $H_{L_{\alpha}}^{p}(\mathbb{R}^{n})$ is the
following BMO space.
\begin{definition}[{\cite[Definition 3.10]{bn22}}]\label{d2.9}
Let $f \in L^{2}_\loc(\mathbb{R}^{n})$ and $\beta\in(0,\infty)$. Then $f$ is said to be of type
$(L_{\alpha}, \beta)$ if	
$$
\int_{\mathbb{R}^{n}} \frac{|f(x)|^{2}}{(1+|x|)^{n+\beta}}\left(1+\frac{1}{|x|}\right)^{\sigma}\,dx<\infty .
$$	
We denote the set of all functions of type $\left(L_{\alpha},\beta\right)$ by $M_{\beta}$.
	
For $f \in M_{\beta}$, we define
$$
\|f\|_{M_{\beta}}:=\left[\int_{\mathbb{R}^{n}} \frac{|f(x)|^{2}}{(1+|x|)^{n+\beta}}
\left(1+\frac{1}{|x|}\right)^{\sigma}\,dx\right]^{\frac{1}{2}} .
$$
It is easy to see that $M_{\beta}$ is a Banach space and $M_{\beta} \subset M_{\beta^{\prime}}$
for $\beta<\beta^{\prime}$.
\end{definition}

\begin{definition}[{\cite[Definition 3.12]{bn22}}]\label{d2.10}
Let $M \in \mathbb{N}$ and $\gamma\in[0,\frac{\alpha}{n})$. We say that $f \in M_{\alpha}$ is in
$\text{BMO}_{L_{\alpha}, M}^{\gamma}(\mathbb{R}^{n})$, if
$$
\|f\|_{\text{BMO}_{L_{\alpha}, M}^{\gamma}(\rn)}:=\sup _{B \in \mathbb{R}^{n}}\left(\frac{1}{|B|^{1+2 \gamma}}
\int_{B}\left|\left(I-e^{-r_{B}^{\alpha} L_{\alpha}}\right)^{M+1} f(x)\right|^{2} d x\right)^{\frac{1}{2}}<\infty,
$$	
where the supremum is taken over all balls $B=B\left(x_{B}, r_{B}\right)$ of $\mathbb{R}^{n}$.
\end{definition}

\begin{theorem}[{\cite[Theorem 3.15]{bn22}}]\label{t2.11}
For any $p \in(\frac{n}{n+\alpha}, 1]$ and $M\in\mathbb{N}$ with
$$M>\max\lf\{\frac{n+2 \alpha}{2 \alpha}, \frac{n}{p \alpha}\r\},$$
the dual space of $H_{L_{\alpha}}^{p}(\mathbb{R}^{n})$ is the space $\mathrm{BMO}_{L_{\alpha}, M}^{\frac{1}{p}-1}
(\mathbb{R}^{n})$ in the following sense.
\begin{itemize}
\item [{\rm (a)}]Suppose $f \in \mathrm{BMO}_{L_{\alpha}, M}^{\frac{1}{p}-1}(\mathbb{R}^{n})$,
then the linear functional given by
\begin{equation}\label{2.2}
\ell(g)=\int_\rn f(x) g(x)\, dx,
\end{equation}
initially defined on $H_{L_{\alpha}}^{p}(\mathbb{R}^{n})\cap L^{2}(\mathbb{R}^{n})$, has a unique bounded extension
to $H_{L_{\alpha}}^{p}(\mathbb{R}^{n})$.
		
\item [{\rm (b)}] Conversely, for every bounded linear functional $\ell$ on $H_{L_{\alpha}}^{p}(\mathbb{R}^{n})$
can be realized as in \eqref{2.2}, that is, there exists $f \in \mathrm{BMO}_{L_{\alpha}, M}^{\frac{1}{p}-1}(\mathbb{R}^{n})$
such that \eqref{2.2} holds and
$$
\|f\|_{\mathrm{BMO}_{L_\alpha,M}^{\frac{1}{p}-1}(\rn)}\le c|\ell|_{H_{L_\alpha}^p(\rn)}
$$
for some positive constant $c$ independent of $\ell$.
\end{itemize}
\end{theorem}

Now, we recall the definition and some basic properties of the tent spaces introduced by Coifman et al. in [8].
We mention that the tent spaces is a useful tool to study the function spaces in harmonic analysis such as
Hardy spaces, BMO spaces, and VMO spaces, etc.
	
For $x \in \mathbb{R}^{n}$ and $\beta\in(0,\infty)$, we denote
$$
\Gamma^{\beta}(x):=\left\{(y, t) \in \mathbb{R}^{n} \times(0, \infty):|x-y| \leq \beta t\right\}.
$$
When $\beta=1$, we briefly write $\Gamma(x)$ instead of $\Gamma^{1}(x)$.
	
For any closed subset $F \subset \mathbb{R}^{n}$, define
$$
R^{\beta}(F):=\cup_{x \in F} \Gamma^{\beta}(x),
$$
and we also denote $R^{1}(F)$ by $R(F)$.
	
If $O$ is an open set in $\mathbb{R}^{n}$, then the tent over $\widehat{O}$ is defined as
$$\widehat{O}:=\left(R\left(O^{c}\right)\right)^{c}.$$
Let $\mathbb{R}_{+}^{n+1}:=\left\{(y, t) \in \mathbb{R}^{n+1}:\ t\in(0,\infty)\right\}$. For a measurable
function $F$ defined in $\mathbb{R}_{+}^{n+1}$, we define
$$
A(F)(x) :=\left(\int_{\Gamma(x)}|F(y, t)|^{2} \frac{d y d t}{t^{n+1}}\right)^{\frac{1}{2}},
$$
$$
C(F)(x) :=\sup _{x \in B}\left(\frac{1}{|B|} \int_{\widehat{B}}|F(y, t)|^{2} \frac{d y d t}{t}\right)^{\frac{1}{2}},
$$
and
$$
C_{p}(F)(x):=\sup _{x \in B} \frac{1}{|B|^{\frac{1}{p}-\frac{1}{2}}}\left(\int_{\widehat{B}}|F(y, t)|^{2}\,
\frac{d y d t}{t}\right)^{\frac{1}{2}} \quad \text { for } p\in(0,1].
$$
	
\begin{definition}[{\cite[Definition 2.4]{bn22}}]\label{d2.12}
The tent spaces are defined as follows. For $p\in(0,\infty)$, we define $T_{2}^{p}(\mathbb{R}_{+}^{n+1}):
=\{F:\ A(F) \in L^{p}(\mathbb{R}^{n})\}$ with the quasi-norm
$\|F\|_{T_{2}^{p}}:=\|A(F)\|_{p}$ which is a Banach space for $ p\in[1,\infty)$.
		
For $p=\infty$, we define $T_{2}^{\infty}(\mathbb{R}_{+}^{n+1}):=\{F:\ C(F) \in L^{\infty}(\mathbb{R}^{n})
\}$ with the norm $\|F\|_{T_{2}^{\infty}}:=\|C(F)\|_{\infty}$ which is a Banach space.
		
For $p\in(0,1]$, we define $T_{2}^{p, \infty}(\mathbb{R}_{+}^{n+1}):=\{F:\ \|C_{p}(F)\|_{\infty}<\infty\}$
with the norm $\|F\|_{T_{2}^{p, \infty}}:=\|C_{p} F\|_{\infty}$. Obviously,
$T_{2}^{1, \infty}(\mathbb{R}_{+}^{n+1})=T_{2}^{\infty}(\mathbb{R}_{+}^{n+1})$.
\end{definition}
	
\begin{proposition}[{\cite[Proposition 2.7]{bn22}}]\label{p2.13}
\begin{itemize}
\item [{\rm (i)}]The following inequality holds, whenever $f \in T_{2}^{1}(\mathbb{R}_{+}^{n+1})$
and $g \in T_{2}^{\infty}(\mathbb{R}_{+}^{n+1})$ :
$$
\int_{\mathbb{R}_{+}^{n+1}}|f(x, t) g(x, t)|\,\frac{d x d t}{t} \leq
C\int_{\mathbb{R}^{n}} A(f)(x) C(g)(x)\,d x .
$$
    		
\item [{\rm (ii)}]Suppose $p\in (1,\infty)$, $f \in T_{2}^{p}(\mathbb{R}_{+}^{n+1})$, and
$g \in T_{2}^{p^{\prime}}(\mathbb{R}_{+}^{n+1})$ with $\frac{1}{p}+\frac{1}{p^{\prime}}=1$, then
$$
\int_{\mathbb{R}_{+}^{n+1}}|f(y, t) g(y, t)|\,\frac{d y d t}{t} \leq
\int_{\mathbb{R}^{n}} A(f)(x) A(g)(x)\,d x .
$$
\item [{\rm (iii)}]If $p \in(0,1]$, then the dual space of $T_{2}^{p}(\mathbb{R}_{+}^{n+1})$ is
$T_{2}^{p, \infty}(\mathbb{R}_{+}^{n+1})$. More precisely, the pairing
$$
\langle f, g\rangle=\int_{\mathbb{R}_{+}^{n+1}} f(x, t) g(x, t)\,\frac{d x d t}{t}
$$
realizes $T_{2}^{p, \infty}(\mathbb{R}_{+}^{n+1})$ as the dual of $T_{2}^{p}(\mathbb{R}_{+}^{n+1})$.
\end{itemize}
\end{proposition}

Recall that a measure $\nu$ is a Carleson measure of order $\beta \in[1,\infty)$, if there is a positive
constant $c$ such that, for each ball $B$ on $\mathbb{R}^{n}$,
\begin{equation}\label{2.3}
\nu(\widehat{B}) \leq c|B|^{\beta} .
\end{equation}
The smallest constant in \eqref{2.3} is define to be the norm of $\nu$, and is denoted by $\|\nu\|_{V^\beta}$;
see, for instance, \cite{t86}.
	
\begin{lemma}\label{l2.14}
Let $M,M_1\in \mathbb{N}$ such that $M\geq M_1\geq \frac{n}{\alpha}(\frac{1}{p}-1)$ and let
$p\in(\frac{n}{n+\alpha}, 1]$. Suppose that $f \in \mathcal{M_\alpha}(\mathbb{R}^{n})$ such that
$$\left|\lf(t^\alpha L_\alpha\r)^M e^{-t^\alpha L_\alpha}\lf(I-e^{-t^\alpha L_\alpha}\r)^{M_1+1}f(x)
\right|^{2} \frac{d x d t}{t}$$ is a Carleson measure of order $\frac{2}{p}-1$ on $\mathbb{R}_{+}^{n+1}$
and $g \in H_{L_\alpha}^p(\mathbb{R}^{n}) \cap L^{2}(\mathbb{R}^{n})$. Then there exists a positive constant $C$ such that
\begin{align*}
&\int_{\mathbb{R}^{n}} f(x) g(x)\,d x\\
&\quad=C \int_{\mathbb{R}_{+}^{n+1}} \lf(t^\alpha L_\alpha\r)^M
e^{-t^\alpha L_\alpha}\lf(I-e^{-t^\alpha L_\alpha}\r)^{M_1+1} f(x) t^\alpha L_\alpha e^{-t^\alpha
L_\alpha}g(x)\, \frac{d x d t}{t} .
\end{align*}
\end{lemma}
\begin{proof}
Let $$G(x,t):=t^\alpha L_\alpha e^{-t^\alpha L_\alpha}g(x)$$ and
$$F(x,t):=\lf(t^\alpha L_\alpha \r)^Me^{-t^\alpha L_\alpha}\lf(I-e^{-t^\alpha L_\alpha}\r)^{M_1+1}f(x).$$
Since $g\in H_{L_\alpha}^p(\mathbb{R}^{n})$, it follows that $S_{L_{\alpha}}(g)(x)\in L^p(\mathbb{R}^{n})$,
which, together with the definition of the tent space $T_2^p(\mathbb{R}_{+}^{n+1})$, implies that
$G\in T_2^p(\mathbb{R}_{+}^{n+1})$. On the other hand, by the assumption that
$$\left|\lf(t^\alpha L_\alpha\r)^M e^{-t^\alpha L_\alpha}\lf(I-e^{-t^\alpha L_\alpha}\r)^{M_1+1}
f(x)\right|^{2}\, \frac{d x d t}{t}$$
is a Carleson measure of order $\frac{2}{p}-1$ on $\mathbb{R}_{+}^{n+1}$ and the definition of the tent space
$T_2^{p,\infty}(\mathbb{R}_{+}^{n+1})$, we find that $F\in T_2^{p,\infty}(\mathbb{R}_{+}^{n+1})$.
From (iii) of Proposition \ref{p2.13}, we deduce that
$$\int_{\mathbb{R}_{+}^{n+1}}|G(x,t)F(x,t)|\,\frac{dxdt}{t}\ls 1.$$
Notice that $L_\alpha$ is a self-adjoint operator. Using the $L^2$-functional calculus, we conclude that
there is a positive constant $C$ such that
$$
g(x)=C\int_0^\infty \lf(t^\alpha L_\alpha\r)^M e^{-t^\alpha L_\alpha}t^\alpha L_\alpha
e^{-t^\alpha L_\alpha}\lf(I-e^{-t^\alpha L_\alpha}\r)^{M_1+1}g(x)\frac{dt}{t}\text{ in }L^2(\mathbb{R}^{n}).
$$
From this, together with Fubini's theorem, it follows that
\begin{align*}
&\int_{\mathbb{R}_{+}^{n+1}}G(x,t)F(x,t)\frac{dxdt}{t}\\
&=\int_0^\infty\int_{\mathbb{R}}\lf(t^\alpha L_\alpha\r)^M e^{-t^\alpha L_\alpha}\lf(I-e^{-t^\alpha L_\alpha}\r)^{M_1+1}
f(x)t^\alpha L_\alpha e^{-t^\alpha L_\alpha}g(x)\frac{dx dt}{t}\\
&=\int_0^\infty\int_{\mathbb{R}}f(x)\lf(t^\alpha L_\alpha\r)^M e^{-t^\alpha L_\alpha}\lf(I-e^{-t^\alpha L_\alpha}\r)
^{M_1+1}t^\alpha L_\alpha e^{-t^\alpha L_\alpha}g(x)\frac{dx dt}{t}\\
&=C^{-1}\int_{\mathbb{R}}f(x)g(x)dx.
\end{align*}	
This completes the proof of Lemma \ref{l2.14}.
	\end{proof}

\begin{lemma}[{\cite[Lemma 3.14]{bn22}}]\label{l2.15}
Let $M,M_1\in  \mathbb{N}$ such that $M\geq M_1$ and let $ \gamma\in[0,\frac{\alpha}{n})$. If $f\in
\mathrm{BMO}_{L_{\alpha}, M_1}^{\gamma}(\mathbb{R}^{n})$, then
\begin{equation*}
v(x,t)=\left|\lf(t^\alpha L_\alpha\r)^M e^{-t^\alpha L_\alpha}\lf(I-e^{-t^\alpha L_\alpha}\r)^{M_1+1}f(x)\right|\,\frac{dxdt}{t}
\end{equation*}
is a Carleson measure of order $2\gamma +1$ with $\|v\|_{V^{2\gamma+1}}\ls \|f\|^2_{
\mathrm{BMO}_{L_{\alpha}, M_1}^{\gamma}(\rn)}$.
\end{lemma}

\begin{lemma}\label{l2.16}
Let $p\in(\frac{n}{n+\alpha},1]$, $M_0 \geq\max\{\frac{n+2\alpha}{2\alpha},\frac{n}{p\alpha}\}$,
$M_0\in \mathbb{N}$, and $M,M_1\in \mathbb{N}$ such that $M\geq M_1\geq \frac{n}{\alpha}(\frac{1}{p}-1)$.
If $f \in \mathcal{M_\alpha}(\mathbb{R}^{n})$ and
$$\left|\lf(t^\alpha L_\alpha\r)^M e^{-t^\alpha L_\alpha}\lf(I-e^{-t^\alpha L_\alpha}\r)^{M_1+1}f(x)
\right|\,\frac{dxdt}{t}$$
is a Carleson measure of order $\frac{2}{p}-1$ on $\mathbb{R}_{+}^{n+1}$, then
$f \in \mathrm{BMO}_{L_{\alpha}, M_0}^{\frac{1}{p}-1}(\mathbb{R}^{n})$ .	
\end{lemma}
\begin{proof}
For any $g(x)\in H_{L_\alpha}^p(\mathbb{R}^n)\cap L^2(\mathbb{R}^n)$, from the proof of
Lemma \ref{l2.14}, it is easy to see that
\begin{align*}
\lf|\int_{\mathbb{R}^{n}} f(x) g(x) d x\r|\ls 1.
\end{align*}
Then, by Theorem \ref{t2.11} and the fact that $H_{L_\alpha}^p(\mathbb{R}^n)\cap L^2(\mathbb{R}^n)$
is dense in the space $H_{L_\alpha}^p(\mathbb{R}^n)$, we have $f\in \lf(H_{L_\alpha}^p(\mathbb{R}^n)\r)^*
=\mathrm{BMO}_{L_\alpha ,M_0}^{\frac{1}{p}-1}(\mathbb{R}^n)$,
which completes the proof of Lemma \ref{l2.16}.
\end{proof}
	
At the end of this section, we will show some conclusions related to tent spaces, which play an important role
in the subsequent proofs. One of the most important properties of tent spaces is the atomic decomposition.
We now recall the definition of atoms in the tent space $T_2^p(\mathbb{R}_{+}^{n+1})$ for $p\in(0,1]$.
\begin{definition}[{\cite{cms85}}]\label{d2.17}
For $p\in(0,1]$, a measurable function $F$ on $\mathbb{R}_{+}^{n+1}$ is said to be a $T_{2}^{p}$-atom if
there exists a ball $B \subset \mathbb{R}^{n}$ such that $F$ is supported in $\widehat{B}$ and
$$
\int_{\widehat{B}}|F(x, t)|^{2} \frac{d x d t}{t} \leq|B|^{1-\frac{2}{p}}.
$$
\end{definition}
	
\begin{lemma}[{\cite[Lemma 2.6]{bn22}}]\label{l2.18}
Let $p\in(0,1]$. For every $F \in T_{2}^{p}(\mathbb{R}_{+}^{n+1})$ there exist a positive constant $C_{(p)}$, a sequence of numbers $\left\{\lambda_{j}\right\}_{j=0}^{\infty}$ and a sequence of $T_{2}^{p}$-atoms $\left\{A_{j}\right\}_{j=0}^{\infty}$ such that
$$
F=\sum_{j=0}^{\infty} \lambda_{j} A_{j} \text { in } T_{2}^{p}(\mathbb{R}_{+}^{n+1}) \text { a.e in } \mathbb{R}_{+}^{n+1}
$$
and
$$\sum_{j=0}^{\infty}\left|\lambda_{j}\right|^{p} \leq C_{(p)}\|F\|_{T_{2}^{p}}.$$
Furthermore, if $F \in T_{2}^{p}(\mathbb{R}_{+}^{n+1}) \cap T_{2}^{2}(\mathbb{R}_{+}^{n+1})$, then the
sum also converges in $T_{2}^{2}(\mathbb{R}_{+}^{n+1})$.
	\end{lemma}
\begin{definition}\label{d2.19}
Let $p\in(0,1]$. The space $\widetilde{{T}_{2}^p}(\mathbb{R}_{+}^{n+1})$ is defined to be the set of all
$f=\sum_{j} \lambda_{j} a_{j}$, where the series converges in
$({T}_{2}^{p,\infty}(\mathbb{R}_{+}^{n+1}))^{*}$, $\left\{a_{j}\right\}_{j}$ are $T_{2}^p$-atoms and
$\{\lambda_{j}\}_{j} \in \ell^{p}$. If $f \in \widetilde{{T}_{2}^p}(\mathbb{R}_{+}^{n+1})$, then define
$$
\|f\|_{\widetilde{{T}_{2}^p}} := \inf \left\{\lf(\sum_{j}\left|\lambda_{j}\right|^p\r)^\frac{1}{p}\right\}
$$
where the infimum is taken over all possible decompositions of $f$ as above.
	\end{definition}
By {\cite[Lemma 3.1]{hm09}}, $\widetilde{{T}_{2}^p}(\mathbb{R}_{+}^{n+1})$ is a Banach space.
Moreover, as a special case of \cite[Lemma 4.1]{jy112}, we have the following Lemma \ref{l2.20}, which implies
that $T_2^p(\mathbb{R}_{+}^{n+1})$ is dense in $\widetilde{{T}_{2}^p}(\mathbb{R}_{+}^{n+1})$. Thus, in this
sense, $\widetilde{{T}_{2}^p}(\mathbb{R}_{+}^{n+1})$ is called the Banach completion of the space
$T_2^p(\mathbb{R}_{+}^{n+1})$.
	
\begin{lemma}\label{l2.20}
There exists a positive constant $C$ such that, for all $f\in T_2^p(\mathbb{R}_{+}^{n+1})$,
$f\in\widetilde{{T}_{2}^p}(\mathbb{R}_{+}^{n+1})$ and
$$
\|f\|_{\widetilde{{T}_{2}^p}}\le C\|f\|_{T_2^p}.$$
\end{lemma}

\begin{lemma}\label{l2.21}
Let $p\in(\frac{n}{n+\alpha},1]$ and $M,M_1\in \mathbb{N}$ with $M\geq 2M_1$. Suppose that $A$ is a
$T_{2}^{p}$-atom supported in $\widehat{B}$ with some ball $B \subseteq \mathbb{R}^{n}$. Then there is a positive constant $c_{(M,M_1)}$ such that the function
$$
c_{(M,M_1)} \int_{0}^{\infty}\lf(t^\alpha L_\alpha\r)^M e^{-t^\alpha L_\alpha}\lf(I-e^{-t^\alpha
L_\alpha}\r)^{2M_1+1}A(x, t)\,\frac{d t}{t}
$$	
define a multiple of $(p, 2, M, \epsilon)_{L_{\alpha}}$-molecule associated to the ball $B$ with
$\epsilon=\alpha+d-\frac{n}{p}$.
\end{lemma}
\begin{proof}
Let $A$ be a $T_2^p$-atom supported in $\widehat{B}$ with some ball $B$. Then we have
\begin{equation}\label{2.4}
\int_{\widehat{B}}|A(x, t)|^{2} \frac{d x d t}{t} \leq|B|^{1-\frac{2}{p}}.
\end{equation}
Now, let
\begin{align*}
m(x)&:=\int_0^\infty \lf(t^\alpha L_\alpha\r)^M e^{-t^\alpha L_\alpha}\lf(I-e^{-t^\alpha L_\alpha}\r)^{2M_1+1}A(x, t) \frac{d t}{t}\\
&=L_\alpha^M\int_0^\infty t^{\alpha M}e^{-t^\alpha L_\alpha}\lf(I-e^{-t^\alpha L_\alpha}\r)^{2M_1+1}A(x, t) \frac{d t}{t}.
\end{align*}
Then $m=L^M_\alpha b$, where
\begin{align*}
b:=\int_0^\infty t^{\alpha M}e^{-t^\alpha L_\alpha}\lf(I-e^{-t^\alpha L_\alpha}\r)^{2M_1+1}A(x, t) \frac{d t}{t}.
\end{align*}
Moreover, for a fixed $i\in \mathbb{N}$, let $f\in L^2(S_i(B))$ satisfies that $\supp (f)\subset S_i(B)$
and $\|f\|_{L^2(S_i(B))}=1$. By the H\"older inequality and \eqref{2.4}, we have
\begin{align*}
\int_{\mathbb{R}^{n}}m(x)\overline{f(x)}dx&=\int_{\mathbb{R}^{n}}\int_{0}^\infty \lf(t^\alpha L_\alpha\r)^M e^{-t^\alpha L_\alpha}\lf(I-e^{-t^\alpha L_\alpha}\r)^{2M_1+1}A(x, t) \overline{f(x)}\frac{d t}{t}dx\\
&\le \int_{\widehat{B}}\lf|A(x,t)\lf(t^\alpha L_\alpha\r)^M e^{-t^\alpha L_\alpha}
\lf(I-e^{-t^\alpha L_\alpha}\r)^{2M_1+1}\overline{f(x)}\r|\frac{d xdt}{t}\\
&\le \lf(\int_{\widehat{B}}|A(x,t)|^2\frac{d t}{t}dx\r)^\frac{1}{2}\\
&\hs\times\lf(\int_{\widehat{B}}\lf|\lf(t^\alpha L_\alpha\r)^M e^{-t^\alpha L_\alpha}\lf(I-
e^{-t^\alpha L_\alpha}\r)^{2M_1+1}\overline{f(x)}\r|^2\frac{d xdt}{t}\r)^\frac{1}{2}\\
&\le |B|^{\frac{1}{2}-\frac{1}{p}}\lf(\int_{\widehat{B}}\lf|\lf(t^\alpha L_\alpha\r)^M e^{-t^\alpha L_\alpha}\lf(I-e^{-t^\alpha L_\alpha}\r)^{2M_1+1}
\overline{f(x)}\r|^2\frac{d xdt}{t}\r)^{\frac{1}{2}}.
\end{align*}
It follows from Proposition \ref{p2.3} and Theorem \ref{t2.2} that
\begin{align*}
&\lf(\int_{\widehat{B}}\lf|\lf(t^\alpha L_\alpha\r)^M e^{-t^\alpha L_\alpha}\lf(I-e^{-t^\alpha L_\alpha}\r)^{2M_1+1}
\overline{f(x)}\r|^2\frac{d xdt}{t}\r)^{\frac{1}{2}}\\
&\quad\ls \sum_{j=0}^{2M_1+1}\lf(\int_{\widehat{B}}\lf|\lf(t^\alpha L_\alpha\r)^M
e^{-t^\alpha L_\alpha}e^{-jt^\alpha L_\alpha}\overline{f(x)}\r|^2\frac{d xdt}{t}\r)^{\frac{1}{2}}\\
&\quad\ls c_{(M,M_1)}\lf(\int_{0}^{r_B}\lf(\frac{2^ir_B}{t}\r)^{2n}\lf(\frac{2^ir_B}{t}\r)
^{-2n-2\alpha}\frac{dt}{t}\r)^\frac{1}{2}2^{-\frac{in}{2}}\|f\|_{L^2(S_i(B))}\\
&\quad\ls c_{(M,M_1)}2^{-i(\alpha+\frac{n}{2})},
\end{align*}
which, together with the above estimate, implies that
\begin{align*}
\int_{\mathbb{R}^{n}}m(x)\overline{f(x)}\,dx\le c_{(M,M_1)}2^{-i(\alpha+n-\frac{n}{p})}|2^iB|
^{\frac{1}{2}-\frac{1}{p}}.
\end{align*}
Taking supremum over all $f$ such that $\|f\|_{L^2(2^iB)}=1$, we obtain
\begin{align*}
\|m\|_{L^2(S_i(B))}\le c_{(M,M_1)}2^{-i(\alpha+n-\frac{n}{p})}|2^iB|^{\frac{1}{2}-\frac{1}{p}}.
\end{align*}
Similarly, we have
\begin{align*}
\|\lf(r_B^\alpha L_\alpha\r)^kb\|_{L^2(2^iB)}\le c_{(M,M_1)}r_B^{\alpha M}2^{-i\ez}|2^iB|
^{\frac{1}{2}-\frac{1}{p}}
\end{align*}
for all integer $k$ such that $k\le M$, where $\ez=\alpha+n-\frac{n}{p}>0$.
This finishes the proof of Lemma \ref{l2.21}.
	\end{proof}
	
\section{The VMO-type space $\mathrm{VMO}_{L_\alpha,M}^\gamma(\rn)$}
In this section, we study spaces of functions with vanishing mean oscillation associated with the operator
$L_\alpha$. We begin with some notions and notation.	
\begin{definition}\label{d3.1}
Let $M \in \mathbb{N}$ and $\gamma\in[0,\frac{\alpha}{n})$. A function $f \in \mathrm{BMO}_{L_{\alpha}, M}^{\gamma}(\mathbb{R}^{n})$
is said to be in $\mathrm{VMO}_{L_{\alpha}, M}^{\gamma}(\mathbb{R}^{n})$, if it satisfies the limiting
conditions $\gamma_{1}(f)=\gamma_{2}(f)=$ $\gamma_{3}(f)=0$, where	
$$
\begin{aligned}
& \gamma_{1}(f) := \lim _{c \rightarrow 0} \sup _{\text {ball } B: r_{B} \leq c}
\left(\frac{1}{|B|^{1+2\gamma}} \int_{B}\left|\lf(I-e^{-r_B^\alpha L_\alpha}\r)^{M+1}f(x)\right|^{2} d x\right)^{1 / 2}, \\
& \gamma_{2}(f) := \lim _{c \rightarrow \infty} \sup _{\text {ball } B: r_{B} \geq c}
\left(\frac{1}{|B|^{1+2\gamma}} \int_{B}\left|\lf(I-e^{-r_B^\alpha L_\alpha}\r)^{M+1}f(x)\right|^{2} d x\right)^{1 / 2}
\end{aligned}$$	
and	
$$
\gamma_{3}(f) := \lim _{c \rightarrow \infty} \sup _{\text {ball } B \subset[B(0, c)]^{\complement}}
\left(\frac{1}{|B|^{1+2\gamma}} \int_{B}\left|\lf(I-e^{-r_B^\alpha L_\alpha}\r)^{M+1}f(x)\right|^{2}\, d x\right)^{1 / 2} .
$$
For any function $f\in\mathrm{VMO}_{L_{\alpha},M}^{\gamma}(\mathbb{R}^{n})$, we define
$\|f\|_{\mathrm{VMO}_{L_{\alpha}, M}^{\gamma}(\rn)} :=\|f\|_{\mathrm{BMO}_{L_{\alpha}, M}^{\gamma}(\rn)}$.
\end{definition}
	
In what follows, let $T_{2, 0}^{p,\infty}(\mathbb{R}_{+}^{n+1})$ be the set of all $f \in T_{2}^{p,\infty}
(\mathbb{R}_{+}^{n+1})$ satisfying $\eta_{1}(f)=$ $\eta_{2}(f)=\eta_{3}(f)=0$, where
$$
\eta_{1}(f) := \lim _{c \rightarrow 0} \sup _{\text {ball } B: r_B \leq c} \frac{1}{|B|^{\frac{1}{p}-1}}
\left(\frac{1}{|B|} \int_{\widehat{B}}|f(y, t)|^{2} \frac{d y d t}{t}\right)^{1 / 2},$$
$$
\eta_{2}(f) := \lim _{c \rightarrow \infty} \sup _{\text {ball } B: r_{B} \geq c}
\frac{1}{|B|^{\frac{1}{p}-1}}\left(\frac{1}{|B|} \int_{\widehat{B}}|f(y, t)|^{2} \frac{d y d t}{t}\right)^{1 / 2}
$$
and
$$
\eta_{3}(f) := \lim _{c \rightarrow \infty}  \sup _{\text {ball }B \subset[B(0, c)]^{\complement}}
\frac{1}{|B|^{\frac{1}{p}-1}}\left(\frac{1}{|B|} \int_{\widehat{B}}|f(y, t)|^{2} \frac{d y d t}{t}\right)^{1 / 2} .
$$
Obviously, $T_{2, 0}^{p,\infty}(\mathbb{R}_{+}^{n+1})$ is a closed linear subspace of $T_{2}^{p,\infty}
(\mathbb{R}_{+}^{n+1})$.
	
Denote by $T_{2, 1}^{p,\infty}(\mathbb{R}_{+}^{n+1})$ the space of all $f \in T_{2}^{p,\infty}
(\mathbb{R}_{+}^{n+1})$ with $\eta_{1}(f)=0$, and $T_{2, c}^{2}(\mathbb{R}_{+}^{n+1})$ the space of all
$f \in T_{2}^{2}(\mathbb{R}_{+}^{n+1})$ with compact support. From \cite[P.\,10]{jy112}, we deduce that
$T_{2, c}^{2}(\mathbb{R}_{+}^{n+1}) \subset T_{2, 0}^{p,\infty}(\mathbb{R}_{+}^{n+1}) \subset
T_{2, 1}^{p,\infty}(\mathbb{R}_{+}^{n+1})$. Moreover, let $T_{2, c}^{p,\infty}(\mathbb{R}_{+}^{n+1})$ denote
the set of all $f \in T_{2}^{p,\infty}(\mathbb{R}_{+}^{n+1})$ with compact support, then $T_{2, c}^{p,\infty}
(\mathbb{R}_{+}^{n+1})$ coincides with $T_{2, c}^{2}(\mathbb{R}_{+}^{n+1})$; see also {\cite[Page 10]{jy112}}.
Define $T_{2, v}^{p,\infty}(\mathbb{R}_{+}^{n+1})$ to be the closure of $T_{2, c}^{p,\infty}(\mathbb{R}_{+}^{n+1})$ in
$T_{2, 1}^{p,\infty}(\mathbb{R}_{+}^{n+1})$.

The following Proposition \ref{p3.2} and Theorem \ref{t3.3} are the corollaries of
\cite[Proposition 3.1 and Theorem 4.2]{jy112}, respectively.
\begin{proposition}\label{p3.2}
Let $T_{2, v}^{p,\infty}(\mathbb{R}_{+}^{n+1})$ and $T_{2, 0}^{p,\infty}(\mathbb{R}_{+}^{n+1})$ be
defined as above. Then $T_{2, v}^{p,\infty}(\mathbb{R}_{+}^{n+1})$ and $T_{2, 0}^{p,\infty}(\mathbb{R}_{+}^{n+1})$
coincide with equivalent norms.
\end{proposition}

\begin{theorem}\label{t3.3}
The space $(T_{2, v}^{p,\infty}(\mathbb{R}_{+}^{n+1}))^{*}$, the dual space of the space
$T_{2, v}^{p,\infty}(\mathbb{R}_{+}^{n+1})$, coincides with $\widetilde{{T}_{2}^p}(\mathbb{R}_{+}^{n+1})$
in the following sense:
\begin{itemize}
\item [{\rm (i)}]For any given $g \in \widetilde{{T}_{2}^p}(\mathbb{R}_{+}^{n+1})$, define the linear functional $\ell$ by setting, for all $f \in$ $T_{2, v}^{p,\infty}(\mathbb{R}_{+}^{n+1})$,    	
\begin{equation}\label{3.1}
\ell(f):= \int_{\mathbb{R}_{+}^{n+1}} f(x, t) g(x, t) \frac{d x d t}{t} .
\end{equation}
Then there exists a positive constant $C$, independent of $g$, such that
$$
\|\ell\|_{\left(T_{2, v}^{p,\infty}\right)^{*}} \leq C\|g\|_{\widetilde{{T}_{2}^p}} .
$$
\item [{\rm (ii)}]Conversely, for any $\ell \in(T_{2, v}^{p,\infty}(\mathbb{R}_{+}^{n+1}))^{*}$,
there exists $g \in \widetilde{{T}_{2}^p}(\mathbb{R}_{+}^{n+1})$ such that \eqref{3.1} holds for all
$f \in T_{2, v}^{p,\infty}(\mathbb{R}_{+}^{n+1})$ and $\|g\|_{\widetilde{{T}_{2}^p}} \leq
C\|\ell\|_{\left(T_{2, v}^{p,\infty}\right)^{*}}$, where $C$ is independent of $\ell$.
\end{itemize}
\end{theorem}

\begin{theorem}\label{t3.4}
Let $p\in(\frac{n}{n+\alpha},1]$. Let $M,M_0,M_1\in \mathbb{N}$ with $M_1 \geq M_{0} \geq\max
\{\frac{n+2\alpha}{2\alpha},\frac{n}{p\alpha}\}$ and $M \geq 2 M_{1}$. Then the following conclusions are equivalent:
\begin{itemize}
\item [{\rm (a)}]$f \in \mathrm{VMO}_{L_{\alpha}, M_0}^{\frac{1}{p}-1}(\mathbb{R}^{n})$,
			
\item [{\rm (b)}]$f \in \mathcal{M_\alpha}(\mathbb{R}^{n})$ and $\lf(t^\alpha L_\alpha\r)^M
e^{-t^\alpha L_\alpha}\lf(I-e^{-t^\alpha L_\alpha}\r)^{2M_1+1}f \in T_{2,v}^{p,\infty}(\mathbb{R}_{+}^{n+1})$.
\end{itemize}
\end{theorem}
\begin{proof}
First, we prove that (a) implies (b). By the fact $\mathrm{VMO}_{L_\alpha,M_0}^{\frac{1}{p}-1}
(\mathbb{R}^{n})\subset \mathrm{BMO}_{L_\alpha,M_0}^{\frac{1}{p}-1}(\mathbb{R}^{n})$ and the definition of
the BMO space, we only need to verify that
$$\lf(t^\alpha L_\alpha\r)^M e^{-t^\alpha L_\alpha}\lf(I-e^{-t^\alpha L_\alpha}\r)^{2M_1+1}f
\in T_{2,v}^{p,\infty}(\mathbb{R}_{+}^{n+1}).$$
To this end, we need to show that, for any ball $B:=B(x_B,r_B)$,
\begin{align}\label{3.2}
&\frac{1}{|B|^{\frac{1}{p}-\frac{1}{2}}}\lf(\int_{\widehat{B}}\lf|\lf(t^\alpha L_\alpha\r)^M
e^{-t^\alpha L_\alpha}\lf(I-e^{-t^\alpha L_\alpha}\r)^{2M_1+1}f(x)\r|^2\frac{dxdt}{t}\r)^\frac{1}{2}\nonumber\\
&\quad\ls \sum_{k=1}^{\infty}2^{-k\alpha}\dz_k (f,B),
\end{align}
where
\begin{align*}
&\dz_k (f,B)\\
&\quad:=\sup_{B^*\subset 2^{k+1}B:r_{B^*}\in[r_B/2,2r_B]}\frac{1}{|B^*|^{\frac{1}{p}-\frac{1}{2}}}\lf(\int_{B^*}\lf|\lf(I-e^{-r_{B^*}^\alpha L_\alpha}\r)^{M_0+1}f(x)\r|^2\frac{dxdt}{t}\r)^\frac{1}{2}.
\end{align*}
In fact, since $f\in \mathrm{VMO}_{L_\alpha,M_0}^{\frac{1}{p}-1}(\mathbb{R}^{n})\subset \mathrm{BMO}_{L_\alpha,M_0}
^{\frac{1}{p}-1}(\mathbb{R}^{n})$, it follows that
$$\dz_k (f,B)\le \|f\|_{\mathrm{BMO}_{L_\alpha,M_0}^{\frac{1}{p}-1}(\rn)}<\infty.$$
Further, for each $k\in \mathbb{N}$, by the definition of $\mathrm{VMO}_{L_\alpha,M_0}^{\frac{1}{p}-1}
(\mathbb{R}^{n})$, we have
\begin{align*}
\lim_{c\rightarrow 0}\sup_{B:r_B\le c} \dz_k (f,B)=\lim_{c\rightarrow \infty}\sup_{B:r_B\geq c} \dz_k (f,B)=\lim_{c\rightarrow \infty}\sup_{B\subset B(0,c)^{\complement}} \dz_k (f,B)=0.
\end{align*}
This, together with \eqref{3.2} and the dominated convergence theorem on series, shows that
\begin{align}\label{3.3}
&\eta_{1}\lf(\lf(t^\alpha L_\alpha\r)^M e^{-t^\alpha L_\alpha} \lf(I-e^{-t^\alpha L_\alpha}\r)^{2M_1+1}f\r)\nonumber\\
&\quad=\lim_{c \rightarrow 0}\sup_{B:r_B\le c}\frac{1}{|B|^{\frac{1}{p}-\frac{1}{2}}}\lf(\int_{\widehat{B}}
\lf|\lf(t^\alpha L_\alpha\r)^M e^{-t^\alpha L_\alpha} \lf(I-e^{-t^\alpha L_\alpha}\r)^{2M_1+1}f(x)\r|^2\frac{dxdt}{t}\r)\nonumber\\
&\quad\ls \sum_{k=1}^\infty 2^{-k\alpha}\lim_{c \rightarrow 0}\sup_{B:r_B\le c} \dz_k (f,B)=0.
\end{align}
Thus,
$$\eta_{1}\lf(\lf(t^\alpha L_\alpha\r)^M e^{-t^\alpha L_\alpha} \lf(I-e^{-t^\alpha L_\alpha}\r)^{2M_1+1}f\r)=0.$$

Similarly, we have
\begin{align*}
&\eta_{2}\lf(\lf(t^\alpha L_\alpha\r)^M e^{-t^\alpha L_\alpha} \lf(I-e^{-t^\alpha L_\alpha}\r)^{2M_1+1}f\r)\\
&\quad=\eta_{3}\lf(\lf(t^\alpha L_\alpha\r)^M e^{-t^\alpha L_\alpha} \lf(I-e^{-t^\alpha L_\alpha}\r)^{2M_1+1}f\r)\\
&\quad=0,
\end{align*}
which, together with Proposition \ref{3.2}, implies that
$$\lf(t^\alpha L_\alpha\r)^M e^{-t^\alpha L_\alpha} \lf(I-e^{-t^\alpha L_\alpha}\r)^{2M_1+1}f\in
T_{2,v}^{p,\infty}(\mathbb{R}_{+}^{n+1}),$$ and hence, (a) implies (b).

Now, we turn to prove \eqref{3.2}. Notice that
\begin{align*}
&\frac{1}{|B|^{\frac{1}{p}-\frac{1}{2}}}\left[\int_{\widehat{B}}\left|\lf(t^\alpha L_\alpha\r)^M e^{-t^\alpha L_\alpha}
\lf(I-e^{-t^\alpha L_\alpha}\r)^{2M_1+1}f(x)\right|^{2} \frac{d x d t}{t}\right]^{1 / 2} \\
&\quad\leq \frac{1}{|B|^{\frac{1}{p}-\frac{1}{2}}}\left[\int_{\widehat{B}} \left|\lf(t^\alpha L_\alpha\r)^M e^{-t^\alpha L_\alpha}
\lf(I-e^{-t^\alpha L_\alpha}\r)^{2M_1+1}\lf(I-e^{-r_{2B}^\alpha L_\alpha}\r)^{M_0+1}f(x)\right|^{2} \frac{d x d t}{t}\right]^{1 / 2} \\
&\hs\quad +\frac{1}{|B|^{\frac{1}{p}-\frac{1}{2}}}\left[\int_{\widehat{B}}\left|\lf(t^\alpha L_\alpha\r)^M e^{-t^\alpha L_\alpha}
\lf(I-e^{-t^\alpha L_\alpha}\r)^{2M_1+1}\right.\right.\\
&\hs\quad\times\left.\left.\lf(I-\lf(I-e^{-r_{2B}^\alpha L_\alpha}\r)^{M_0+1}\r)f(x)\right|^{2} \frac{d x d t}{t}\right]^{1 / 2} \\
&\quad=: \mathrm{I}+\mathrm{J} .
\end{align*}

To estimate the term $\mathrm{I}$, set $U_{1}(B) := 2 B$ and $U_{k}(B) :=(2^{k} B)\backslash (2^{k-1} B)$
when $k \geq 2$. For $k \in \mathbb{N}$, let
$$b_{k} :=\left[\lf(I-e^{-r_{2B}^\alpha L_\alpha}\r)^{M_0+1}f\right] \chi_{U_{k}(B)}.$$
Then, by the Minkowski inequality, we conclude that
\begin{align*}
\mathrm{I} &=\frac{1}{|B|^{\frac{1}{p}-\frac{1}{2}}}\left(\int_{\widehat{B}} \left|\lf(t^\alpha L_\alpha\r)^M
e^{-t^\alpha L_\alpha} \lf(I-e^{-t^\alpha L_\alpha}\r)^{2M_1+1}\right.\right.\\
&\hs\times\left.\left.\sum_{k=1}^\infty\lf[\lf(I-e^{-r_{2B}^\alpha L_\alpha}\r)^{M_0+1}f(x)\r]\chi_{U_{k}(B)}
\right|^{2} \frac{d x d t}{t}\right)^{1 / 2} \\
&\le \sum_{k=1}^\infty\frac{1}{|B|^{\frac{1}{p}-\frac{1}{2}}}\left(\int_{\widehat{B}} \left|
\lf(t^\alpha L_\alpha\r)^M e^{-t^\alpha L_\alpha} \lf(I-e^{-t^\alpha L_\alpha}\r)^{2M_1+1}b_k(x)\right|^{2} \frac{d x d t}{t}\right)^{1 / 2} \\
&=: \sum_{k=1}^{\infty} \mathrm{I}_{k}
\end{align*}

When $k=1$, from Theorem \ref{t2.2} and Proposition \ref{p2.3}, we deduce that
\begin{align}\label{3.4}
\mathrm{I}_{1} &=\frac{1}{|B|^{\frac{1}{p}-\frac{1}{2}}}\left(\int_{\widehat{B}} \left|\lf(t^\alpha L_\alpha\r)^M
e^{-t^\alpha L_\alpha} \lf(I-e^{-t^\alpha L_\alpha}\r)^{2M_1+1}b_1(x)\right|^{2} \frac{d x d t}{t}\right)^\frac{1}{2}\nonumber\\
&=\frac{1}{|B|^{\frac{1}{p}-\frac{1}{2}}}\left(\int_{\widehat{B}} \left|\sum_{i=0}^{2M_1+1}C_i^{2M_1+1}
\lf(t^\alpha L_\alpha\r)^Me^{-(i+1)t^\alpha L_\alpha}b_1(x)\right|^{2} \frac{d x d t}{t}\right)^\frac{1}{2}\nonumber\\
&\ls \frac{1}{|B|^{\frac{1}{p}-\frac{1}{2}}}\sum_{i=0}^{2M_1+1}\lf(\int_0^{r_B}|B|\lf(\frac{2^jr_B}{t}\r)^{2n}
\lf(\frac{2^jr_B}{t}\r)^{-2n-2\alpha}\lf(\fint_{S_{j}(B)}|b_1(x)|^2dx\r)\frac{dt}{t}\r)^\frac{1}{2}\nonumber\\
&\ls \frac{1}{|B|^{\frac{1}{p}-\frac{1}{2}}}\sum_{i=0}^{2M_1+1}\lf(\int_0^{r_B}\frac{|B|}{|2^jB|}
\lf(\frac{2^jr_B}{t}\r)^{-2\alpha}\|b_1\|_{L^2{(S_j(B))}}^2\frac{dt}{t}\r)^\frac{1}{2}\nonumber\\
&\ls \frac{1}{|B|^{\frac{1}{p}-\frac{1}{2}}}\|b_1\|_{L^2{(S_j(B))}}\nonumber\\
&\ls\dz_1(f,B),
\end{align}
here and hereafter, $C_i^{2M_1+1}:=(-1)^i {i \choose 2M_1+1}$.

When $k \geq 2$, notice that for any $x \in B$ and $y \in(2^{k} B)^{\complement},|x-y|
\gtrsim 2^{k} r_{B}$. It then follows from Proposition \ref{p2.3} that
\begin{align*}
&\left|\lf(t^\alpha L_\alpha\r)^M e^{-t^\alpha L_\alpha} \lf(I-e^{-t^\alpha L_\alpha}\r)^{2M_1+1}b_k(x)\right|\\
&\quad= \left|\sum_{i=0}^{2M_1+1}C_i^{2M_1+1}\lf(t^\alpha L_\alpha\r)^Me^{-(i+1)t^\alpha L_\alpha}b_k(x)\right|\\
&\quad\ls\int_{U_{k}(B)}t^{-n}D_\sz(x,t^\alpha)D_\sz(y,t^\alpha) \lf(\frac{t+|x-y|}{t}\r)^{-n-\alpha}\lf|b_k(y)\r|dy \\
&\quad\ls \lf(2^{k} r_B\r)^{-n-\alpha}t^\alpha D_\sz(x,t^\alpha)\int_{2^{k} B}D_\sz(y,t^\alpha)\lf|
\left(I-e^{-r_{2B}^\alpha L_\alpha}\right)^{M_0+1} f(y)\r|dy .
\end{align*}

For any ball $B(x_{B}, 2^{k} r_{B})$ with $k \geq 2$, there exists a collection $\{B_{k, 1},
B_{k, 2}, \ldots, B_{k, N_{k}}\}$ of balls such that each ball $B_{k, i}\subset 2^{k+1}B$ and
is of radius $r_{2 B}$, $B(x_{B}, 2^{k} r_{B}) \subset \bigcup_{i=1}^{N_{k}} B_{k, i}$ with $N_{k}
\lesssim 2^{n k}$; see, for instance, \cite[pp. 645-646]{ddy05}. By this, the H\"older inequality,
and Lemma \ref{l2.1}, we find that, for $t\in(0,r_B)$,
\begin{align*}
&\left|\lf(t^\alpha L_\alpha\r)^M e^{-t^\alpha L_\alpha} \lf(I-e^{-t^\alpha L_\alpha}\r)^{2M_1+1}b_k(x)\right|\\
&\quad\ls\sum_{i=1}^{N_{k}} \lf(2^{k} r_B\r)^{-n-\alpha}t^\alpha D_\sz(x,t^\alpha)\int_{B_{k,i}}D_\sz(y,t^\alpha)
\lf|\left(I-e^{-r_{2B}^\alpha L_\alpha}\right)^{M_0+1} f(y) \r|dy \\
&\quad\ls \sum_{i=1}^{N_{k}} \lf(2^{k} r_B\r)^{-n-\alpha}t^\alpha D_\sz(x,t^\alpha)
\lf(\int_{B_{k,i}}\lf|\left(I-e^{-r_{2B}^\alpha L_\alpha}\right)^{M_0+1} f(y)\r|^2 dy\r)^\frac{1}{2}\\
&\qquad\times\lf(\int_{B_{k,i}} D_{2\sz}(y,t^\alpha) d y\r)^\frac{1}{2}\\
&\quad\ls N_{k} \lf(2^{k} r_B\r)^{-n-\alpha}t^\alpha D_\sz(x,t^\alpha)|B_{k,i}|^{\frac{1}{p}-\frac{1}{2}}
\dz_k(f,B)r_B^{\frac{n}{2}}\\
&\quad\ls2^{-k\alpha}r_B^{-n-\alpha+\frac{n}{p}}t^\alpha D_\sz(x,t^\alpha) \delta_{k}(f, B).
\end{align*}
From this and Lemma \ref{l2.1}, we deduce that
\begin{align}\label{3.5}
\mathrm{I}_{k}&=\frac{1}{|B|^{\frac{1}{p}-\frac{1}{2}}}\left(\int_{\widehat{B}}\left|\lf(t^\alpha L_\alpha\r)^M
e^{-t^\alpha L_\alpha} \lf(I-e^{-t^\alpha L_\alpha}\r)^{2M_1+1}b_k(x)\right|^{2} \frac{d x d t}{t}\right)^{\frac{1}{2}}\nonumber\\
&\ls \frac{1}{|B|^{\frac{1}{p}-\frac{1}{2}}}\left(\int_0^{r_B}\int_{B}
2^{-2k\alpha}r_B^{-2n-2\alpha+2\frac{n}{p}}t^{2\alpha
} D_{2\sz}(x,t^\alpha) \delta_{k}^2(f, B)\frac{d x d t}{t}\right)^{\frac{1}{2}}\nonumber\\
&\ls \frac{1}{|B|^{\frac{1}{p}-\frac{1}{2}}}2^{-k\alpha}r_B^{-n-\alpha+\frac{n}{p}} \delta_{k}(f, B)
\left(\int_0^{r_B} t^{2\alpha}\int_{B}D_{2\sz}(x,t^\alpha)dx\frac{d t}{t}\right)^{\frac{1}{2}}\nonumber\\
&\ls r_B^{\frac{n}{2}-\frac{n}{p}}2^{-k\alpha}r_B^{-n-\alpha+\frac{n}{p}} r_B^{\frac{n}{2}}\left(\int_0^{r_B}
t^{2\alpha}\frac{d t}{t}\right)^{\frac{1}{2}}\delta_{k}(f, B)\nonumber\\
&\ls 2^{-k\alpha}\delta_{k}(f, B).
\end{align}
Combining \eqref{3.4} and \eqref{3.5}, we obtain that $I\lesssim \sum_{k=1}^{\infty} 2^{-k \alpha}
\delta_{k}(f, B)$.

To estimate the term $\mathrm{J}$, write
\begin{align*}
&\lf(I-e^{-t^\alpha L_\alpha}\r)^{2M_1+1}\\
&\hs=\lf(I-e^{-r_{2B}^\alpha L_\alpha}+e^{-r_{2B}^\alpha L_\alpha}-e^{-t^\alpha L_\alpha}\r)^{2M_1+1}\\
&\hs=\lf[\sum_{k=0}^{M_1}C_k^{2M_1+1}e^{-kt^\alpha L_\alpha}\lf(I-e^{-r_{2B}^\alpha L_\alpha}\r)^{2M_1-M_0-k}
\lf(I-e^{-(r_{2B}^\alpha-t^\alpha)L_\alpha}\r)^k\r]\\
&\hs\hs\times\lf(I-e^{-r_{2B}^\alpha L_\alpha}\r)^{M_0+1}\\
&\hs\hs+\lf[\sum_{k=M_1+1}^{2M_1+1}C_k^{2M_1+1}e^{-kt^\alpha L_\alpha}\lf(I-e^{-r_{2B}^\alpha L_\alpha}
\r)^{2M_1-k+1}\lf(I-e^{-(r_{2B}^\alpha-t^\alpha)L_\alpha}\r)^{k-M_0-1}\r]\\
&\hs\hs\times\lf(I-e^{-(r_{2B}^\alpha-t^\alpha)L_\alpha}\r)^{M_0+1}\\
&\hs=:\Psi_1(L_\alpha)\lf(I-e^{-r_{2B}^\alpha L_\alpha}\r)^{M_0+1}+\Psi_2(L_\alpha)\lf(I-e^{-(r_{2B}^\alpha
-t^\alpha)L_\alpha}\r)^{M_0+1}.
\end{align*}
For $t \in\left(0, r_{B}\right)$, by Proposition \ref{p2.3}, we conclude that the following upper bound for
the kernel of $(t^\alpha L_\alpha)^M e^{-t^\alpha L_\alpha}(I-(I-e^{-r_{2B}^\alpha L_\alpha})^{M_0+1})
\Psi_{1}(L_\alpha)$:
\begin{align*}
\left(\frac{t}{r_{B}}\right)^{\alpha M}D_\sz(x,r_B^\alpha)D_\sz(y,r_B^\alpha)
\frac{\left(r_{B}\right)^{\alpha}}{\left(r_{B}+|x-y|\right)^{n+\alpha}},\ \forall x, y \in \mathbb{R}^{n}.
\end{align*}
Using this and some computation similar to the estimate for $\mathrm{I}_{k}$, we find that, for any
$x \in B$ and $t \in\left(0, r_{B}\right)$,
\begin{align*}
&\left|\lf(t^\alpha L_\alpha\r)^M e^{-t^\alpha L_\alpha} \lf(I-\lf(I-e^{-r_{2B}^\alpha L_\alpha}\r)^{M_0+1}\r)\Psi_1(L_\alpha)\lf(I-e^{-r_{2B}^\alpha L_\alpha}\r)^{M_0+1}f(x)\right|\\
&\quad\ls\sum_{k=1}^{\infty} \left|\lf(t^\alpha L_\alpha\r)^M e^{-t^\alpha L_\alpha}
\lf(I-\lf(I-e^{-r_{2B}^\alpha L_\alpha}\r)^{M_0+1}\r)\r.\\
&\qquad\lf.\times\Psi_1(L_\alpha)\lf[\lf(I-e^{-r_{2B}^\alpha L_\alpha}\r)^{M_0+1}f\chi_{U_{k}(B)}\r](x)\right|\\
&\quad\ls\sum_{k=1}^{\infty} \int_{U_{k}(B)}\left(\frac{t}{r_{B}}\right)^{\alpha M}D_\sz(x,r_B^\alpha)
D_\sz(y,r_B^\alpha) \frac{\left(r_{B}\right)^{\alpha}}{\left(r_{B}+|x-y|\right)^{n+\alpha}}\\
&\qquad\times\lf|\lf(I-e^{-r_{2B}^\alpha L_\alpha}\r)^{M_0+1}f(y)\r|dy\\
&\quad\ls \left(\frac{t}{r_{B}}\right)^{\alpha M} r_B^\alpha \sum_{k=1}^{\infty}(2^kr_B)^{-n-\alpha}D_\sz(x,r_B^\alpha)\\
&\qquad\times\sum_{i=1}^{N_k}\int_{B_{k,i}}D_\sz(y,r_B^\alpha)\lf|\lf(I-e^{-r_{2B}^\alpha L_\alpha}\r)^{M_0+1}f(y)\r|dy\\
&\quad\ls\left(\frac{t}{r_{B}}\right)^{\alpha M} r_B^\alpha \sum_{k=1}^{\infty}(2^kr_B)^{-n-\alpha}D_\sz(x,r_B^\alpha)N_k
|B_{k,i}|^{\frac{1}{p}-\frac{1}{2}}\delta_{k}(f, B)r_B^{\frac{n}{2}}\\
&\quad\ls\sum_{k=1}^{\infty}2^{-k\alpha}r_B^{-\alpha M-n-\frac{n}{p}}t^{\alpha M}D_\sz(x,r_B^\alpha)\delta_{k}(f, B).
\end{align*}
It then follows from Lemma \ref{l2.1} that
\begin{align*}
&\frac{1}{|B|^{\frac{1}{p}-\frac{1}{2}}}\left(\int_{\widehat{B}} \left|\lf(t^\alpha L_\alpha\r)^M e^{-t^\alpha L_\alpha} \lf(I-\lf(I-e^{-r_{2B}^\alpha L_\alpha}\r)^{M_0+1}\r)\Psi_1(L_\alpha)\right.\right.\\
&\qquad\left.\left.\times\lf(I-e^{-r_{2B}^\alpha L_\alpha}\r)^{M_0+1}f(x)\right|^{2} \frac{d x d t}{t}\right)^{\frac{1}{2}}\\
&\quad\ls r_B^{\frac{n}{2}-\frac{n}{p}}\sum_{k=1}^{\infty}2^{-k\alpha}r_B^{-\alpha M-n-\frac{n}{p}}\delta_{k}(f, B)\left(\int_0^{r_B} t^{2\alpha M
}\int_{B}D_{2\sz}(x,r_B^\alpha)dx\frac{d t}{t}\right)^{\frac{1}{2}}\\
&\quad\ls \sum_{k=1}^\infty 2^{-k\alpha}\delta_{k}(f, B).
\end{align*}
Similarly, we have that, for all $x \in B$ and $t \in\left(0, r_{B}\right)$,
\begin{align*}
&\left|\lf(t^\alpha L_\alpha\r)^M e^{-t^\alpha L_\alpha} \lf(I-\lf(I-e^{-r_{2B}^\alpha L_\alpha}\r)^{M_0+1}\r)\r.\\
&\qquad\times\lf.\Psi_2(L_\alpha)\lf(I-e^{-(r_{2B}^\alpha-t^\alpha) L_\alpha}\r)^{M_0+1}f(x)\right|\\
&\quad\ls\sum_{k=1}^{\infty}2^{-k\alpha}r_B^{-\alpha M-n-\frac{n}{p}}t^{\alpha M}D_\sz(x,r_B^\alpha)\delta_{k}(f, B)
\end{align*}
and
\begin{align*}
&\frac{1}{|B|^{\frac{1}{p}-\frac{1}{2}}}\left(\int_{\widehat{B}} \left|\lf(t^\alpha L_\alpha\r)^M e^{-t^\alpha L_\alpha} \lf(I-\lf(I-e^{-r_{2B}^\alpha L_\alpha}\r)^{M_0+1}\r)\Psi_2(L_\alpha)\right.\right.\\
&\qquad\left.\left.\times\lf(I-e^{-(r_{2B}^\alpha-t^\alpha) L_\alpha}\r)^{M_0+1}f(x)\right|^{2} \frac{d x d t}{t}\right)^{\frac{1}{2}}\\
&\quad\ls \sum_{k=1}^\infty 2^{-k\alpha}\delta_{k}(f, B).
\end{align*}

Combining the above two estimates, we have
$$
\mathrm{J} \lesssim \sum_{k=1}^\infty 2^{-k \alpha} \delta_{k}(f, B),
$$
which, together with the estimate for $\mathrm{I}$, implies that \eqref{3.2} holds. This finishes the proof
of the conclusion that (a) implies (b).

Now we turn to prove that (b) implies (a). By (b) and the definition of $T_{2}^{p,\infty}(\mathbb{R}_{+}^{n+1})$
space, we conclude that
$$\lf|\lf(t^\alpha L_\alpha\r)^M e^{-t^\alpha L_\alpha}\lf(I-e^{-t^\alpha L_\alpha}\r)^{2M_1+1}f\r|^2\frac{dxdt}{t}$$
is a Carleson measure. It then follows from Lemma \ref{l2.16} that $f \in \mathrm{BMO}_{L_\alpha,M_0}^{\frac{1}{p}-1}(\mathbb{R}^{n})$.
Notice that
\begin{align*}
&\frac{1}{|B|^{\frac{1}{p}-\frac{1}{2}}}\left(\int_{{B}} \left|\lf(I-e^{-r_{B}^\alpha L_\alpha}\r)^{M_0+1}f(x)\right|^2dx\right)^\frac{1}{2} \\
&\quad=\sup _{\|g\|_{L^{2}(\mathbb{R}^n)} \leq 1} \frac{1}{|B|^{\frac{1}{p}-\frac{1}{2}}}\left|\int_{\mathbb{R}^n} f(x)\lf(I-e^{-r_{B}^\alpha L_\alpha}\r)^{M_0+1}
(g\chi_B) (x) d x\right| .
\end{align*}
Let $g\in L^2(\mathbb{R}^n)$. Using \cite[Theorems 3.2 and 3.3]{bd21}, we deduce that
\begin{align*}
&\lf(\int_{\mathbb{R}^{n}}\lf|\lf(I-e^{-r_{B}^\alpha L_\alpha}\r)^{M_0+1} (g\chi_B
)(x)\r|^2dx\r)^\frac{1}{2}\\
&\quad\ls \sum_{i=0}^{M_0+1}\lf(\int_{\mathbb{R}^{n}}\lf|C_i^{M_0+1}e^{-ir_{B}^\alpha L_\alpha}(g\chi_B
) (x) \r|^2dx\r)^\frac{1}{2}\\
&\quad\ls \|g\|_{L^2(\mathbb{R}^n)}< \infty.
\end{align*}
Moreover, by \cite[Lemma 3.16]{bn22}, we have
\begin{align*}
	\lf\|S_{L_\alpha}\lf(\lf(I-e^{-r_{B}^\alpha L_\alpha}\r)^{M_0+1} (g\chi_B
	)\r)\r\|_{L^p(\mathbb{R}^n)}\ls \|B\|^{\frac{1}{p}-\frac{1}{2}} \|g \chi_B\|_{L^2(\mathbb{R}^n)}< \infty.
\end{align*}
Thus,
$$\lf(I-e^{-r_{B}^\alpha L_\alpha}\r)^{M_0+1} (g\chi_B)(x) \in H^p_{L_\alpha}(\mathbb{R}^n)\cap
L^2(\mathbb{R}^n).$$
This, together with Lemma \ref{l2.14}, shows that
\begin{align*}
&\left|\int_{\mathbb{R}^n} f(x)\lf(I-e^{-r_{B}^\alpha L_\alpha}\r)^{M_0+1} (g\chi_B)(x) d x\right| \\
&\quad=\left|\int_{\mathbb{R}_{+}^{n+1}} \lf(t^\alpha L_\alpha\r)^M e^{-t^\alpha L_\alpha}\lf(I-e^{-t^\alpha L_\alpha}\r)^{2M_1+1} f(x)t^\alpha L_\alpha e^{-t^\alpha L_\alpha}\r.\\
&\hs\quad\times\lf.\lf(I-e^{-r_{B}^\alpha L_\alpha}\r)^{M_0+1} (g\chi_B) (x) \frac{d x d t}{t}\right| \\
&\quad\ls\int_{\widehat{4 B}}\left| \lf(t^\alpha L_\alpha\r)^M e^{-t^\alpha L_\alpha}\lf(I-e^{-t^\alpha
L_\alpha}\r)^{2M_1+1} f(x)\r.\\
&\qquad\lf. \times t^\alpha L_\alpha e^{-t^\alpha L_\alpha}\lf(I-e^{-r_{B}^\alpha L_\alpha}\r)^{M_0+1} (g\chi_B
)(x)  \right|\frac{d x d t}{t} \\
&\hs\quad+\sum_{k=2}^{\infty} \int_{\widehat{2^{k+1} B}  \backslash \widehat{{2}^k B}} \cdots \\
&\quad=: \mathrm{A}_{1}+\sum_{k=2}^{\infty} \mathrm{A}_{k} .
\end{align*}

For $A_1$, by an argument similar to the estimate for $\mathrm{I_1}$, we have
\begin{align*}
&\left(\int_{\widehat{4B}} \left|t^\alpha L_\alpha e^{-t^\alpha L_\alpha} \lf(I-e^{-t^\alpha L_\alpha}\r)^{M_0+1}(g\chi_B
)(x) \right|^{2} \frac{d x d t}{t}\right)^\frac{1}{2}\nonumber\\
&\quad=\left(\int_{\widehat{4B}} \left|\sum_{i=0}^{M_0+1}C_i^{M_0+1}t^\alpha L_\alpha e^{-(i+1)t^\alpha L_\alpha}(g\chi_B
)(x) \right|^{2} \frac{d x d t}{t}\right)^\frac{1}{2}\nonumber\\
&\quad\ls \sum_{i=0}^{M_0+1}\lf[\int_0^{4r_B}|4B|\lf(\frac{2^j4r_B}{t}\r)^{2n}\lf(\frac{2^j4r_B}{t}\r)^{-2n-2\alpha}\r.\\
&\qquad\lf.\times\lf(\fint_{S_{j}(4B)}|(g\chi_B)(x) |^2\,dx\r)\frac{dt}{t}\r]^\frac{1}{2}\nonumber\\
&\ls \sum_{i=0}^{M_0+1}\lf[\int_0^{4r_B}\frac{|4B|}{|2^j4B|}\lf(\frac{2^j4r_B}{t}\r)^{-2\alpha}\|g
\chi_B \|_{L^2{(\mathbb{R}^n)}}^2\frac{dt}{t}\r]^\frac{1}{2}\nonumber\\
&\ls \|g\|_{L^2{(B)}}\ls 1,
\end{align*}
which, together with the H\"older inequality, further implies that
\begin{align*}
\mathrm{A}_{1} &\lesssim \left(\int_{\widehat{4 B}}\left| \lf(t^\alpha L_\alpha\r)^M e^{-t^\alpha L_\alpha}\lf(I-e^{-t^\alpha
L_\alpha}\r)^{2M_1+1} f(x)\right|^{2} \frac{d x d t}{t}\right)^{\frac{1}{2}} \\
&\hs\times\left(\int_{\widehat{4 B}}\left|t^\alpha L_\alpha e^{-t^\alpha L_\alpha}\lf(I-e^{-r_{B}^\alpha
L_\alpha}\r)^{M_0+1} (g\chi_B)(x)  \right|^{2} \frac{d x d t}{t}\right)^\frac{1}{2} \\
& \lesssim\left(\int_{\widehat{4 B}}\left| \lf(t^\alpha L_\alpha\r)^M e^{-t^\alpha L_\alpha}\lf(I-e^{-t^\alpha
L_\alpha}\r)^{2M_1+1} f(x)\right|^{2} \frac{d x d t}{t}\right)^\frac{1}{2}.
\end{align*}

When $k \geq 2$, notice that
\begin{align*}
& t^\alpha L_\alpha e^{-t^\alpha L_\alpha}\lf(I-e^{-r_{B}^\alpha L_\alpha}\r)^{M_0+1} (g\chi_B)(x) \\
&\quad=\int_{0}^{r_{B}^{\alpha}} \ldots\times \int_{0}^{r_{B}^{\alpha}} t^\alpha L_\alpha^{M_0+2}
e^{-(t^\alpha+r_1+\cdots+r_{M_0+1}) L_\alpha} (g\chi_B) d r_{1} \cdots d r_{M_{0}+1} .
\end{align*}
Let $\widetilde{t}:=t^\alpha+r_1+\cdots+r_{M_0+1}$. Applying Proposition \ref{p2.3}, we get the following
upper bound for the kernel of $L_\alpha^{M_0+2} e^{-(t^\alpha+r_1+\cdots+r_{M_0+1}) L_\alpha}$:
\begin{align*}
\widetilde{t}^{-(M_0+2+\frac{n}{\alpha})}D_\sz(x,\widetilde{t})
D_\sz(y,\widetilde{t})\lf(\frac{\widetilde{t}+|x-y|}{\widetilde{t}}^\frac{1}{\alpha}\r)^{-n-\alpha}, \ \forall x,y\in\mathbb{R}^n.
\end{align*}
Since $(x, t) \in \widehat{2^{k+1} B} \backslash\widehat{2^{k} B}$ and $y \in B$, we have
$|x-y|+\widetilde{t}^\frac{1}{\alpha}\gs 2^{k} r_{B}$. By Lemma \ref{l2.1} and the H\"older inequality,
we have that, for any $(x, t) \in \widehat{2^{k+1} B} \backslash \widehat{2^{k} B}$,
\begin{align*}
&\left|t^\alpha L_\alpha e^{-t^\alpha L_\alpha}\lf(I-e^{-r_{B}^\alpha L_\alpha}\r)^{M_0+1} (g\chi_B
	)(x)\right| \\
&\quad\lesssim \frac{1}{(2^kr_B)^{n+\alpha}}\int_{0}^{r_{B}^{\alpha}} \cdots \int_{0}^{r_{B}^{\alpha}} t^\alpha \widetilde{t}^{-M_0-1}D_\sz(x,\widetilde{t})
\int_{B}D_\sz(y,\widetilde{t}) |(g\chi_B) (y)| d y d r_{1} \cdots d r_{M_{0}+1} \\
&\quad\lesssim \frac{1}{(2^kr_B)^{n+\alpha}}\int_{0}^{r_{B}^{\alpha}} \cdots \int_{0}^{r_{B}^{\alpha}}
t^\alpha \widetilde{t}^{-M_0-1}D_\sz(x,\widetilde{t})\lf(\widetilde{t}^{\frac{n}{2\alpha}}\|g\|_{L^2(B)}+\|g\|_{L^1(B)}\r) d r_{1} \cdots d r_{M_{0}+1} \\
&\quad\lesssim \frac{1}{(2^kr_B)^{n+\alpha}}\int_{0}^{r_{B}^{\alpha}} \cdots \int_{0}^{r_{B}^{\alpha}}
t^\alpha \widetilde{t}^{-M_0-1}D_\sz(x,\widetilde{t})\lf(\widetilde{t}^{\frac{n}{2\alpha}}+|B|^\frac{1}{2}\r) d r_{1} \cdots d r_{M_{0}+1},
\end{align*}
which, together with Lemma \ref{l2.1}, implies that
\begin{align*}
&\lf(\int_{{\widehat{2^{k+1}B}}\backslash{\widehat{2^{k}B}}}\left|t^\alpha L_\alpha e^{-t^\alpha L_\alpha}\lf(I-e^{-r_{B}^\alpha L_\alpha}\r)^{M_0+1} (g\chi_B
)(x)\right|^2\frac{dxdt}{t}\r)^\frac{1}{2}\\
&\quad\ls \lf[\int_{{\widehat{2^{k+1}B}}\backslash{\widehat{2^{k}B}}}(2^kr_B)^{-2n-2\alpha}\r.\\
&\qquad\times\lf.\lf(\int_{[0,r_B^\alpha]^{M_0+1}}t^\alpha \widetilde{t}^{-M_0-1}\lf(\widetilde{t}^\frac{n}{2\alpha}+|B|^\frac{1}{2}\r)D_\sz(x,\widetilde{t})d\vec{r}\r)^2\frac{dxdt}{t}\r]^\frac{1}{2}\\
&\quad\ls (2^kr_B)^{-n-\alpha}\\
&\qquad\times\int_{[0,r_B^\alpha]^{M_0+1}}\lf(\int_{{\widehat{2^{k+1}B}}\backslash{\widehat{2^{k}B}}}\lf|t^{2\alpha} \widetilde{t}^{-2M_0-2}\lf(\widetilde{t}^\frac{n}{2\alpha}+|B|^\frac{1}{2}\r)^2D_{2\sz}(x,\widetilde{t})\r|\frac{dxdt}{t}\r)^\frac{1}{2}d\vec{r}\\
&\quad\ls 2^{k(-n-\alpha)}2^{(k+1)n}r_B^{-\frac{n}{2}-\alpha}\\
&\qquad\times\int_{[0,r_B^\alpha]^{M_0+1}}\lf(\int_0^{2^{k+1} r_B}\lf|t^{2\alpha} \widetilde{t}^{-2M_0-2}\lf(\widetilde{t}^\frac{n}{\alpha}+2\widetilde{t}^\frac{n}{2\alpha}r_B^{\frac{n}{2}}+r_B^n\r)\r|\frac{dt}{t}\r)^\frac{1}{2}d\vec{r}\\
&\quad\ls 2^{k(-n-\alpha)}2^{(k+1)n}r_B^{-\frac{n}{2}-\alpha}\\
&\qquad\times\int_{[0,r_B^\alpha]^{M_0+1}}\lf[\int_0^{2^{k+1} r_B}\lf(t^{-2M_0\alpha+n}+t^{-2M_0\alpha+\frac{n}{2}}r_B^\frac{n}{2}+t^{-2M_0\alpha}r_B^n\r)\frac{dt}{t}\r]^\frac{1}{2}d\vec{r}\\
&\quad\ls 2^{k(-n-\alpha)}2^{(k+1)n} r_B^{-\frac{n}{2}-\alpha}\\
&\qquad\times\int_{[0,r_B^\alpha]^{M_0+1}}\lf[\lf(2^{k+1}r_B\r)^{-2M_0\alpha+n}\r]^\frac{1}{2}d\vec{r}\\
&\quad\ls 2^{-k\alpha(1+M_0)}.
\end{align*}
It then follows from the H\"older inequality that
\begin{align*}
\mathrm{A}_{k} &\lesssim  2^{-k\alpha(1+M_0)}\left(\int_{{\widehat{2^{k+1}B}}\backslash{\widehat{2^{k}B}}}\left|
\lf(t^\alpha L_\alpha\r)^M e^{-t^\alpha L_\alpha}\lf(I-e^{-t^\alpha L_\alpha}\r)^{2M_1+1} f(x)\right|^{2} \frac{d x d t}{t}\right)^{\frac{1}{2}} \\
&\ls 2^{k(\frac{n}{2}-\frac{n}{p}-\alpha-\alpha M_0)}\frac{|B|^{\frac{1}{p}-\frac{1}{2}}}{|2^{k+1}B|^{\frac{1}{p}-\frac{1}{2}}}
\left(\int_{{\widehat{2^{k+1}B}}\backslash{\widehat{2^{k}B}}}\left|\lf(t^\alpha L_\alpha\r)^M\r.\r.\\
&\quad\times\lf.\lf.e^{-t^\alpha L_\alpha}\lf(I-e^{-t^\alpha L_\alpha}\r)^{2M_1+1} f(x)\right|^{2}
\frac{d x d t}{t}\right)^{\frac{1}{2}}.
\end{align*}
Combining the estimates of $\mathrm{A}_{k}$, we finally obtain that
\begin{align*}
&\frac{1}{|B|^{\frac{1}{p}-\frac{1}{2}}}\left(\int_{B}\left|\lf(I-e^{-r_{B}^\alpha L_\alpha}\r)^{M_0+1}f(x)\right|^{2} d x\right)^{\frac{1}{2}}\\
&\quad\lesssim \frac{1}{|4B|^{\frac{1}{p}-\frac{1}{2}}}\left(\int_{\widehat{4 B}}\left| \lf(t^\alpha L_\alpha\r)^M e^{-t^\alpha L_\alpha}
\lf(I-e^{-t^\alpha L_\alpha}\r)^{2M_1+1} f(x)\right|^{2} \frac{d x d t}{t}\right)^\frac{1}{2}\\
&\quad\hs +\sum_{k=2}^{\infty} 2^{k(\frac{n}{2}-\frac{n}{p}-\alpha-\alpha M_0)} \sigma_{k}(f, B),
\end{align*}
where
\begin{align*}
\sigma_{k}(f, B):=& \frac{1}{|2^{k+1}B|^{\frac{1}{p}-\frac{1}{2}}}\left(\int_{{\widehat{2^{k+1}B}}
\backslash{\widehat{2^{k}B}}}\left|\lf(t^\alpha L_\alpha\r)^M e^{-t^\alpha L_\alpha}\r.\r.\\
&\times\lf.\lf.\lf(I-e^{-t^\alpha L_\alpha}\r)^{2M_1+1} f(x)\right|^{2} \frac{d x d t}{t}\right)^{\frac{1}{2}}.
\end{align*}

The assumption that
$$\lf(t^\alpha L_\alpha\r)^M e^{-t^\alpha L_\alpha}\lf(I-e^{-t^\alpha L_\alpha}\r)^{2M_1+1}f\in
T_{2,v }^{p,\infty}(\mathbb{R}_{+}^{n+1}),$$ together with Proposition \ref{p3.2}, shows that
\begin{align*}
&\lim _{c \rightarrow 0} \sup _{B: r_{B} \leq c} \frac{1}{|4B|^{\frac{1}{p}-\frac{1}{2}}}\left(\int_{\widehat{4 B}}\left| \lf(t^\alpha L_\alpha\r)^M
e^{-t^\alpha L_\alpha}\lf(I-e^{-t^\alpha L_\alpha}\r)^{2M_1+1} f(x)\right|^{2} \frac{d x d t}{t}\right)^\frac{1}{2}\\
&\quad=\lim _{c \rightarrow \infty} \sup _{B: r_{B} \geq c} \frac{1}{|4B|^{\frac{1}{p}-\frac{1}{2}}}\left(\int_{\widehat{4 B}}
\left| \lf(t^\alpha L_\alpha\r)^M e^{-t^\alpha L_\alpha}\lf(I-e^{-t^\alpha L_\alpha}\r)^{2M_1+1} f(x)\right|^{2} \frac{d x d t}{t}\right)^\frac{1}{2}\\
&\quad=\lim _{c \rightarrow \infty} \sup _{B \subset B(0, c)^{\complement}} \frac{1}{|4B|^{\frac{1}{p}-\frac{1}{2}}}
\left(\int_{\widehat{4 B}}\left| \lf(t^\alpha L_\alpha\r)^M e^{-t^\alpha L_\alpha}\lf(I-e^{-t^\alpha L_\alpha}\r)^{2M_1+1} f(x)\right|^{2}
\frac{d x d t}{t}\right)^\frac{1}{2}\\
&\quad=0
\end{align*}
and, for each $k \in \mathbb{N}$,
$$
\lim _{c \rightarrow 0} \sup _{B: r_{B} \leq c} \sigma_{k}(f, B)=\lim _{c \rightarrow \infty}
\sup _{B: r_{B} \geq c} \sigma_{k}(f, B)=\lim _{c \rightarrow \infty} \sup _{B \subset B(0, c)^{\complement}}
\sigma_{k}(f, B)=0,
$$
By an argument similar to the proof of \eqref{3.3}, we obtain that
$$\gamma_{1}(f)=\gamma_{2}(f)=\gamma_{3}(f)=0,$$
which implies that $f \in \mathrm{VMO}_{L_\alpha,M_0}^{\frac{1}{p}-1}(\mathbb{R}^{n})$. This finishes the proof
of Theorem \ref{t3.4}.
\end{proof}
	
\section{Preduality between VMO-type spaces $\mathrm{VMO}_{L_\alpha,M_0}^{\frac{1}{p}-1}(\mathbb{R}^{n})$
and Hardy spaces $H^p_{L_\alpha}(\mathbb{R}^{n})$}

In this section, we focus on the main results of the article.
\begin{theorem}\label{t4.1}
For any $p\in(\frac{n}{n+\alpha},1]$, $M_0\in \mathbb{N}$ with $M_0>\max\{\frac{n}{2\alpha}+1,
\frac{n}{p\alpha}\}$, the dual space of  $\mathrm{VMO}_{L_{\alpha}, M_0}^{\frac{1}{p}-1}(\mathbb{R}^{n})$
is the Hardy space $H_{L_\alpha}^p(\mathbb{R}^{n})$ in the following sense:
\begin{itemize}
\item [{\rm (i)}]For any given $f \in H_{L_\alpha}^p(\mathbb{R}^{n})$, define the linear functional $\ell$ by setting,
for all $g \in$ $\mathrm{VMO}_{L_{\alpha}, M_0}^{\frac{1}{p}-1}\left(\mathbb{R}^{n}\r)$,		
\begin{equation}\label{4.1}
\ell(f) :=\langle f, g\rangle .
\end{equation}	
Then there exists a positive constant $C$, independent of $f$, such that
$$\|\ell\|_{(\mathrm{VMO}_{L_{\alpha}, M_0}^{\frac{1}{p}-1}(\mathbb{R}^{n}))^*}\le
C\|f\|_{H_{L_\alpha}^p\left(\mathbb{R}^{n}\r)}.$$
		
\item [{\rm (ii)}]Conversely, for any $\ell \in(\mathrm{VMO}_{L_{\alpha}, M_0}^{\frac{1}{p}-1}(\mathbb{R}^{n})
)^{*}$, there exists a $f \in H_{L_\alpha}^p(\mathbb{R}^{n})$ such that \eqref{4.1} holds and there exists a
positive constant $C$, independent of $\ell$, such that
$$\|f\|_{H_{L_\alpha}^p(\mathbb{R}^{n})} \leq
C\|\ell\|_{(\operatorname{VMO}_{L_{\alpha}, M_0}^{\frac{1}{p}-1}(\mathbb{R}^{n}))^{*}}.$$
\end{itemize}

\end{theorem}

To prove Theorem \ref{t4.1}, we need the following technical lemmas.

Let $M_0\in \mathbb{N}$ and $p\in(\frac{n}{n+\alpha},1]$. Define the operator $\pi_{L_\alpha,M_0}$ by setting,
for all $f \in T_2^p(\mathbb{R}_{+}^{n+1})\cap T_2^2(\mathbb{R}_{+}^{n+1})$ and $x \in \mathbb{R}^{n}$,
$$
\pi_{L_\alpha,M_0}(f)(x) := C_{(\alpha,M_0)} \int_{0}^{\infty} \lf(t^\alpha L_\alpha\r)^{M_0} e^{-t^\alpha
L_\alpha} f(x, t) \frac{d t}{t};
$$
see, for instance, \cite[Page 20]{bn22}. Let $M,M_1\in \mathbb{N}$ with $M \geq M_{1}$. We also define the
operator $\pi_{L_\alpha,M,M_1}$ by setting, for all $f \in T_2^p(\mathbb{R}_{+}^{n+1})\cap T_2^2(\mathbb{R}_{+}^{n+1})$
and $x \in \mathbb{R}^{n}$,
$$
\pi_{L_\alpha,M,M_1}(f)(x) := c_{(M,M_1)} \int_{0}^{\infty} \lf(t^\alpha L_\alpha\r)^M e^{-t^\alpha L_\alpha}
\lf(I-e^{-t^\alpha L_\alpha}\r)^{M_1+1} f(x, t) \frac{d t}{t}.
$$

Recalled that the space $\widetilde{{T}_{2}^p}(\mathbb{R}_{+}^{n+1})$ defined by Definition \ref{d2.19} is the set of all functions $f$ have the expression of infinite sum of $T_2^p$-atoms. As a corollary of \cite[Lemma 4.7]{jy112}, we have the following result.
\begin{lemma}\label{l4.2}
	$T_{2, c}^{2}(\mathbb{R}_{+}^{n+1})$ is dense in $\widetilde{{T}_{2}^p}(\mathbb{R}_{+}^{n+1})$.
\end{lemma}

Applying this, we can deduce the following conclusions.
\begin{lemma}\label{l4.3}
Let $p\in(\frac{n}{n+\alpha},1]$ and $M,M_1\in \mathbb{N}$ with $M \geq M_{1} \geq \frac{n}{\alpha}
(\frac{1}{p}-1)$. Then the operator $\pi_{L_\alpha,M,M_1}$, initially defined on
$T_2^p(\mathbb{R}_{+}^{n+1})\cap T_2^2(\mathbb{R}_{+}^{n+1})$, extends to a bounded linear operator:
\begin{itemize}
\item[{\rm (i)}]from $T_{2}^{p}(\mathbb{R}_{+}^{n+1})$ to $H_{L_{\alpha}}^{p}(\mathbb{R}^{n})$,
			
\item[{\rm (ii)}]from $\widetilde{{T}_{2}^p}(\mathbb{R}_{+}^{n+1})$ to $H_{L_{\alpha}}^{p}(\mathbb{R}^{n})$.
\end{itemize}
\end{lemma}
\begin{proof}
We first prove (i). For $F\in T_2^p(\mathbb{R}_{+}^{n+1})\cap T_2^2(\mathbb{R}_{+}^{n+1})$, we have
$F=\sum_{i=0}^\infty\lz_iA_i$, where $A_i$ are $T_2^p$-atom and $\lf(\sum_{i}|\lz_i|^p\r)^\frac{1}{p}\ls \|F\|_{T_2^p}$. Hence,
\begin{align*}
\pi_{L_\alpha,M,M_1}(F)&=c_{(M,M_1)}\int_{0}^\infty\lf(t^\alpha L_\alpha\r)^M e^{-t^\alpha L_\alpha}
\lf(I-e^{-t^\alpha L_\alpha}\r)^{M_1+1} F(\cdot, t) \frac{d t}{t}\\
&=\sum_{i=0}^\infty \lz_ic_{(M,M_1)}\int_{0}^\infty\lf(t^\alpha L_\alpha\r)^M e^{-t^\alpha L_\alpha}
\lf(I-e^{-t^\alpha L_\alpha}\r)^{M_1+1} A_i(\cdot, t) \frac{d t}{t}\\
&=:\sum_{i=0}^\infty \lz_i m_i.
\end{align*}
From Lemma \ref{l2.21}, we deduce that $m_i$ is a $(p, 2, M, \epsilon)_{L_{\alpha}}$-molecule with
$\ez=\alpha+n-\frac{n}{p}>0$, which further implies that
\begin{align*}
\|\pi_{L_\alpha,M,M_1}(F)\|_{H_{L_\alpha}^p(\mathbb{R}^n)}\le\lf(\sum_{i=0}^\infty|\lz_i|^p\r)^\frac{1}{p}
\ls \|F\|_{T_2^p}.
\end{align*}
	
Now we turn to prove (ii). Let $f \in T_{2, c}^{2}(\mathbb{R}_{+}^{n+1})$. Notice that $T_{2, c}^{2}
(\mathbb{R}_{+}^{n+1})\subset T_2^p(\mathbb{R}_{+}^{n+1})$. It then follows from (i) that
$\pi_{L_\alpha, M, M_{1}}(f) \in H_{L_\alpha}^p(\mathbb{R}^n)$. Moreover, since $T_2^p(\mathbb{R}_{+}^{n+1})
\subset \widetilde{{T}_2^p}(\mathbb{R}_{+}^{n+1})$, by Definition \ref{d2.19}, there exist $T_2^p$-atoms
$\{a_{j}\}_{j}$ and $\{\lambda_{j}\}_{j} \subset \mathbb{C}$ such that $f=\sum_{j} \lambda_{j}
a_{j}$, where the series converges in $(T_2^{p,\infty}(\mathbb{R}_{+}^{n+1}))^{*}$, and
$(\sum_{j}|\lambda_{j}|^p)^\frac{1}{p} \lesssim\|f\|_{\widetilde{{T}_2^p}}$. Thus, for any
$g \in \mathrm{BMO}_{L_\alpha,M_1}^{\frac{1}{p}-\frac{1}{2}}(\mathbb{R}^{n})$, by Theorem \ref{t2.11}
and Lemma \ref{l2.15}, we obtain that
\begin{align*}
\left\langle\pi_{L_\alpha, M, M_{1}}(f), g\right\rangle & =\int_{\mathbb{R}_{+}^{n+1}} f(x, t) \lf(t^\alpha L_\alpha\r)^M e^{-t^\alpha L_\alpha}
\lf(I-e^{-t^\alpha L_\alpha}\r)^{M_1+1}g(x) \frac{d x d t}{t} \\
& =\sum_{j} \lambda_{j} \int_{\mathbb{R}_{+}^{n+1}} a_{j}(x, t) \lf(t^\alpha L_\alpha\r)^M
e^{-t^\alpha L_\alpha}\lf(I-e^{-t^\alpha L_\alpha}\r)^{M_1+1}g(x) \frac{d x d t}{t} \\
& =\sum_{j} \lambda_{j}\left\langle\pi_{L_\alpha, M, M_{1}}\left(a_{j}\right), g\right\rangle,
\end{align*}
which implies that
$$\pi_{L_\alpha, M, M_{1}}(f)=\sum_{j} \lambda_{j} \pi_{L_\alpha, M, M_{1}}\left(a_{j}\right)$$
holds in $(\mathrm{BMO}_{L_\alpha,M_1}^{\frac{1}{p}-\frac{1}{2}}(\mathbb{R}^{n}))^{*}$ and hence,
$$
\left\|\pi_{L_\alpha, M, M_{1}}(f)\right\|_{H_{L_{\alpha}}^{p}(\mathbb{R}^{n})} \leq \lf(\sum_{j}
\left|\lambda_{j}\right|^p\r)^\frac{1}{p} \lesssim\|f\|_{\widetilde{{T}_2^p}}.
$$
By Lemma \ref{l4.2}, we have that $T_{2, c}^{2}(\mathbb{R}_{+}^{n+1})$ is dense in $\widetilde{{T}_2^p}
(\mathbb{R}_{+}^{n+1})$. Using this and a density argument, we deduce that (ii) holds, which completes the
proof of Lemma \ref{l4.3}.
\end{proof}

\begin{lemma}\label{l4.4}
Let $p\in(\frac{n}{n+\alpha},1]$ and $M_0\in \mathbb{N}$ with $M_0>\frac{n}{p\alpha}$. Then the operator $\pi_{L_\alpha,M_0}$, initially defined on $T_2^p(\mathbb{R}_{+}^{n+1})\cap T_2^2(\mathbb{R}_{+}^{n+1})$, extends to a bounded linear operator:
\begin{itemize}
\item[{\rm (i)}] from $T_{2}^{2}(\mathbb{R}_{+}^{n+1})$ to $L^{2}(\mathbb{R}^{n})$,
		
\item[{\rm (ii)}] from $T_{2, v}^{p,\infty}(\mathbb{R}_{+}^{n+1})$ to $\mathrm{VMO}_{L_{\alpha}, M_0}^{\frac{1}{p}-1}(\mathbb{R}^{n})$.
\end{itemize}
\end{lemma}

\begin{proof}
It was showed in \cite[Page 20]{bn22} that (i) holds.

Let us now prove (ii). Let $M,M_1\in \mathbb{N}$ with $M_1\geq \frac{n}{2\alpha}-1$ and $M\geq 2M_1$.
Suppose that $f\in T_{2,v}^{p,\infty}\lf(\mathbb{R}_+^{n+1}\r)$. To show that $\pi_{L_\alpha,M_0}(f)\in
\mathrm{VMO}_{L_\alpha,M_0}^{\frac{1}{p}-1}\lf(\mathbb{R}^{n}\r)$, by Theorem \ref{t3.4}, we only need to
prove that
$$\lf(t^\alpha L_\alpha\r)^M e^{-t^\alpha L_\alpha} \lf(I-e^{-t^\alpha L_\alpha}\r)^{2M_1+1}
\pi_{L_\alpha,M_0}(f)\in T_{2,v}^{p,\infty}\lf(\mathbb{R}_+^{n+1}\r).$$
We first claim that, for any ball $B$,
\begin{align}\label{4.2}
&\frac{1}{|B|^{\frac{1}{p}-\frac{1}{2}}}\lf(\int_{\widehat{B}}\lf(t^\alpha L_\alpha\r)^M e^{-t^\alpha L_\alpha}
\lf(I-e^{-t^\alpha L_\alpha}\r)^{2M_1+1}\pi_{L_\alpha,M_0}(f)\frac{dxdt}{t}\r)^\frac{1}{2}\nonumber\\
&\quad\ls \sum_{k=1}^\infty 2^{-k\alpha M-\frac{kn}{p}-\frac{kn}{2}}\dz_k(f,B),
\end{align}
where
\begin{align*}
\dz_k(f,B):=\frac{1}{|2^kB|^{\frac{1}{p}-\frac{1}{2}}}\lf(\int_{\widehat{2^{k+1}B}}
|f(x,t)|^2\frac{dxdt}{t}\r)^\frac{1}{2}.
\end{align*}
	
To prove \eqref{4.2}, let $f_1:=f \chi_{\widehat{4B}}$ and $f_2:=f \chi_{(\widehat{4B})^{\complement}}$. Then
\begin{align*}
&\lf(\int_{\widehat{B}}|\lf(t^\alpha L_\alpha\r)^M e^{-t^\alpha L_\alpha} \lf(I-e^{-t^\alpha L_\alpha}
\r)^{2M_1+1}\pi_{L_\alpha,M_0}(f)|^2\frac{dxdt}{t}\r)^\frac{1}{2}\\
&\quad\le \sum_{i=1}^2\lf(\int_{\widehat{B}}|\lf(t^\alpha L_\alpha\r)^M e^{-t^\alpha L_\alpha}
\lf(I-e^{-t^\alpha L_\alpha}\r)^{2M_1+1}\pi_{L_\alpha,M_0}(f_i)|^2\frac{dxdt}{t}\r)^\frac{1}{2}\\
&\quad=:\sum_{i=1}^2 \mathrm{I}_i.
\end{align*}
For the term $\mathrm{I}_1$, by Proposition \ref{p2.3}, Theorem \ref{t2.2} and (i), we obtain
\begin{align*}
\mathrm{I}_1& =\left(\int_{\widehat{B}} \left|\lf(t^\alpha L_\alpha\r)^M e^{-t^\alpha L_\alpha}
\lf(I-e^{-t^\alpha L_\alpha}\r)^{2M_1+1}\pi_{L_\alpha,M_0}(f_1) \right|^{2} \frac{d x d t}{t}\right)^\frac{1}{2}\nonumber\\
&=\left(\int_{\widehat{B}} \left|\sum_{i=0}^{2M_1+1}C_i^{2M_1+1}\lf(t^\alpha
L_\alpha\r)^Me^{-(i+1)t^\alpha L_\alpha}\pi_{L_\alpha,M_0}(f_1) \right|^{2} \frac{d x d t}{t}\right)^\frac{1}{2}\nonumber\\
&\ls \sum_{i=0}^{2M_1+1}\lf(\int_0^{r_B}|B|\lf(\frac{2^jr_B}{t}\r)^{2n}\lf(\frac{2^jr_B}{t}\r)^{-2n-2\alpha}\lf(\fint_{S_{j}(B)}|\pi_{L_\alpha,M_0}(f_1)|^2dx\r)\frac{dt}{t}\r)^\frac{1}{2}\nonumber\\
&\ls \sum_{i=0}^{2M_1+1}\lf(\int_0^{r_B}\frac{|B|}{|2^jB|}\lf(\frac{2^jr_B}{t}\r)^{-2\alpha}\|\pi_{L_\alpha,M_0}(f_1) \|_{L^2{(\mathbb{R}^n)}}^2\frac{dt}{t}\r)^\frac{1}{2}\nonumber\\
&\ls \|\pi_{L_\alpha,M_0}(f_1)\|_{L^2\lf(\mathbb{R}^{n}\r)}\ls \|f_1\|_{T_2^2}\sim\|A(f_1)\|_{L^2\lf(\mathbb{R}^{n}\r)}\\
&\sim\lf(\int_{\mathbb{R}^{n}}\int_{\Gamma(x)}|f_1(y,t)|^2\frac{dydt}{t^{n+1}}dx\r)^\frac{1}{2}\sim\lf(\int_{\mathbb{R}^{n}}\int_0^\infty\int_{|x-y|\le t}|f_1(y,t)|^2\frac{dydt}{t^{n+1}}dx\r)^\frac{1}{2}\\
&\sim\lf(\int_{\mathbb{R}^{n}}\int_0^\infty\int_{|x-y|\le t}|f(y,t)\chi_{\widehat{4B}}|^2\frac{dxdt}{t^{n+1}}dy\r)^\frac{1}{2}\\
&\ls\lf(\int_{\widehat{4B}}\lf(\int_{|x-y|\le t}1dx\r)|f(y,t)|^2\frac{dydt}{t^{n+1}}\r)^\frac{1}{2}\\
&\ls \lf(\int_{\widehat{4B}}|f(y,t)|^2\frac{dydt}{t}\r)^\frac{1}{2}\sim|2B|^{\frac{1}{p}-\frac{1}{2}}\dz_1 (f,B).
\end{align*}
To estimate $\mathrm{I}_2$,	using the definition of $\pi_{L_\alpha,M_0}(f)$, we write
\begin{align*}
&\lf(t^\alpha L_\alpha\r)^M e^{-t^\alpha L_\alpha}\lf(I-e^{-t^\alpha L_\alpha}\r)^{2M_1+1}\pi_{L_\alpha,M_0}(f_2)(x)\\
&\hs=\int_{0}^\infty \lf(t^\alpha L_\alpha\r)^M e^{-t^\alpha L_\alpha}\lf(I-e^{-t^\alpha L_\alpha}\r)^{2M_1+1}
\lf(s^\alpha L_\alpha\r)^{M_0} e^{-s^\alpha L_\alpha}(f_2)(x)\frac{ds}{s}\\
&\hs=\int_{0}^\infty t^{\alpha M}s^{\alpha {M_0}} L_\alpha^{M+{M_0}} e^{-(t^\alpha+s^\alpha) L_\alpha}
\sum_{j=0}^{2M_1+1}C_j^{2M_1+1}e^{-jt^\alpha L_\alpha}f_2(x)\frac{ds}{s}\\
&\hs= \sum_{j=0}^{2M_1+1}C_j^{2M_1+1}\int_{0}^\infty t^{\alpha M}s^{\alpha {M_0}} L_\alpha^{M+{M_0}}
e^{-((j+1)t^\alpha+s^\alpha)}f_2(x)\frac{ds}{s}.
\end{align*}
By Proposition \ref{p2.3}, the kernel of the operator $t^{\alpha M}s^{\alpha {M_0}} L_\alpha^{M+{M_0}}
e^{-((j+1)t^\alpha+s^\alpha)}$ has the following upper bound that, for all $x,y\in\mathbb{R}^n$,
\begin{align*}
t^{\alpha M}s^{\alpha {M_0}}(t^\alpha+s^\alpha)^{-(M+M_0+\frac{n}{\alpha})}D_\sz(x,t^\alpha+s^\alpha)
D_\sz(y,t^\alpha+s^\alpha)\lf[\frac{(t^\alpha+s^\alpha)^\frac{1}{\alpha}+|x-y|}{(t^\alpha
+s^\alpha)^\frac{1}{\alpha}}\r]^{-n-\alpha}.
\end{align*}
Notice that, for $k\geq 2$, $ (x,t)\in{\widehat{B}}$ and $(y,s)\in{\widehat{2^{k+1}B}}\backslash
{\widehat{2^{k}B}}$, we have $|x-y|\sim 2^kr_B$ and $(t^\alpha+s^\alpha)^\frac{1}{\alpha}+|x-y|\sim 2^kr_B$.
From this,  the H\"older inequality and Lemma \ref{l2.1}, we deduce that
\begin{align*}
\mathrm{I}_2&\ls\sum_{k=2}^\infty \lf(\int_{\widehat{B}}\lf[\int_{{\widehat{2^{k+1}B}}\backslash{\widehat{2^{k}B}}}t^{\alpha M}s^{\alpha {M_0}}(2^{k\alpha}r_B^\alpha)^{-(M+M_0+\frac{n}{\alpha})}D_\sz(x,2^{k\alpha}r_B^\alpha)\r.\r.\\
&\hs\lf.\lf.\times D_\sz(y,2^{k\alpha}r_B^\alpha)|f(y,s)|\frac{dyds}{s}\r]^2\frac{dxdt}{t}\r)^\frac{1}{2}\\
&\ls \sum_{k=2}^\infty \lf(\int_{\widehat{B}}\lf[\int_{{\widehat{2^{k+1}B}}\backslash{\widehat{2^{k}B}}}t^{2\alpha M}s^{2\alpha {M_0}}(2^{2k\alpha}r_B^{2\alpha})^{-(M+M_0+\frac{n}{\alpha})}D_{2\sz}(x,2^{k\alpha}r_B^\alpha)\r.\r.\\
&\hs\lf.\lf.\times D_{2\sz}(y,2^{k\alpha}r_B^\alpha)\frac{dyds}{s}\int_{\widehat{2^{k+1}B}}|f(y,s)|^2\frac{dyds}{s}\r]\frac{dxdt}{t}\r)^\frac{1}{2}\\
&\ls \sum_{k=2}^\infty \lf(\int_{\widehat{B}}\lf[\int_{{\widehat{2^{k+1}B}}\backslash{\widehat{2^{k}B}}}s^{2\alpha {M_0}}D_{2\sz}(y,2^{k\alpha}r_B^\alpha)\frac{dyds}{s}\r]t^{2\alpha M}D_{2\sz}(x,2^{k\alpha}r_B^\alpha)\frac{dxdt}{t}\r)^\frac{1}{2}\\
&\hs\times\lf(2^{k\alpha}r_B^\alpha\r)^{-(M+M_0+\frac{n}{\alpha})}|2^kB|^{\frac{1}{p}-\frac{1}{2}}\dz_k(f,B)\\
&\ls \sum_{k=2}^\infty 2^{k\alpha M_0+\frac{kn}{2}}r_B^{\alpha M_0+\frac{n}{2}}2^{-k\alpha M-k\alpha M_0-kn}r_B^{-\alpha M-\alpha M_0-n}2^{\frac{kn}{p}-\frac{kn}{2}}r_B^{\frac{n}{p}-\frac{n}{2}}\dz_k(f,B)\\
&\hs\times\lf(\int_{0}^{r_B}t^{2\alpha M-1}\lf(\lf(2^{k\alpha}r_B^\alpha\r)^\frac{n}{\alpha}+r_B^n\r)dt\r)^\frac{1}{2}\\
&\ls \sum_{k=2}^\infty 2^{-k\alpha M-\frac{kn}{p}-\frac{kn}{2}}|B|^{\frac{1}{p}-\frac{1}{2}}\dz_k(f,B),
	\end{align*}
which yields that \eqref{4.2} holds.

For $f\in T_{2,v}^{p,\infty}\lf(\mathbb{R}_+^{n+1}\r)$, by Proposition \ref{p3.2}, we have
$$
\eta_{1}(f) = \lim _{c \rightarrow 0} \sup _{\text {ball } B: r_B \leq c} \frac{1}
{|B|^{\frac{1}{p}-\frac{1}{2}}}\left(\int_{\widehat{B}}|f(y, t)|^{2} \frac{d y d t}{t}\right)^{1 / 2}=0,$$
$$\eta_{2}(f) = \lim _{c \rightarrow \infty} \sup _{\text {ball } B: r_{B} \geq c} \frac{1}{|B|^{\frac{1}{p}
-\frac{1}{2}}}\left( \int_{\widehat{B}}|f(y, t)|^{2} \frac{d y d t}{t}\right)^{1 / 2}=0,
$$
and
$$
\eta_{3}(f) = \lim _{c \rightarrow \infty}  \sup _{\text {ball }B \subset[B(0, c)]^{\complement}}
\frac{1}{|B|^{\frac{1}{p}-\frac{1}{2}}}\left( \int_{\widehat{B}}|f(y, t)|^{2} \frac{d y d t}{t}\right)^{1 / 2}
=0.
$$
Furthermore, for each $k\in\mathbb{N}$, we have
\begin{align*}
\lim_{c\rightarrow 0}\sup_{B:r_B\le c} \dz_k (f,B)=\lim_{c\rightarrow \infty}\sup_{B:r_B\geq c}
\dz_k (f,B)=\lim_{c\rightarrow \infty}\sup_{B\subset B(0,c)^{\complement}} \dz_k (f,B)=0.
\end{align*}
Then, from \eqref{4.2} and the dominated convergence theorem for series, it follows that
\begin{align*}
&\eta_{1}(\lf(t^\alpha L_\alpha\r)^M e^{-t^\alpha L_\alpha} \lf(I-e^{-t^\alpha L_\alpha}\r)^{2M_1+1}
\pi_{L_\alpha,M_0}(f))\\
&\quad\ls \sum_{k=1}^\infty 2^{-k\alpha M-\frac{kn}{p}-\frac{kn}{2}}\lim_{c \rightarrow 0}\sup_{B:r_B\le c}
\dz_k (f,B)=0,
\end{align*}
and
\begin{align*}
&\eta_{2}(\lf(t^\alpha L_\alpha\r)^M e^{-t^\alpha L_\alpha} \lf(I-e^{-t^\alpha L_\alpha}\r)^{2M_1+1}
\pi_{L_\alpha,M_0}(f))\\
&\quad=\eta_{3}(\lf(t^\alpha L_\alpha\r)^M e^{-t^\alpha L_\alpha} \lf(I-e^{-t^\alpha L_\alpha}\r)^{2M_1+1}
\pi_{L_\alpha,M_0}(f))=0,
	\end{align*}
which implies that
$$\lf(t^\alpha L_\alpha\r)^M e^{-t^\alpha L_\alpha} \lf(I-e^{-t^\alpha L_\alpha}\r)^{2M_1+1}
\pi_{L_\alpha,M_0}(f)(x)\in T_{2,v}^{p,\infty}\lf(\mathbb{R}_+^{n+1}\r).$$
This finishes the proof of Lemma \ref{l4.4}.
\end{proof}

\begin{lemma}\label{l4.5}
Let $p\in(\frac{n}{n+\alpha},1]$ and $M_0\in \mathbb{N}$ with $M_{0} \geq\max\{\frac{n+2\alpha}{2\alpha},
\frac{n}{p\alpha}\}$. Then the space $\mathrm{VMO}_{L_{\alpha}, M_0}^{\frac{1}{p}-1}(\mathbb{R}^{n})
\cap L^{2}(\mathbb{R}^{n})$ is dense in $\mathrm{VMO}_{L_{\alpha}, M_0}^{\frac{1}{p}-1}(\mathbb{R}^{n})$.
\end{lemma}
\begin{proof}
Let $M,M_1\in \mathbb{N}$ with $M_1\geq M_{0}$ and $M\geq 2M_1+M_0-1$. Notice that, for any $f\in L^2(\mathbb{R}^n)$, we have
\begin{align*}
f=C\int_{0}^\infty \lf(t^\alpha L_\alpha\r)^M e^{-t^\alpha L_\alpha}\lf(I-e^{-t^\alpha L_\alpha}\r)^{2M_1+1}
t^\alpha L_\alpha e^{-t^\alpha L_\alpha}f(x)\frac{dt}{t}
\end{align*}
in $L^2(\mathbb{R}^n)$, where $C$ is a positive constant.
	
Let $g\in\mathrm{VMO}_{L\alpha,M_0}^{\frac{1}{p}-1}(\mathbb{R}^n)$. By Theorem \ref{t3.4}, we find that
$$\lf(t^\alpha L_\alpha\r)^{M-M
_0+1} e^{-t^\alpha L_\alpha}\lf(I-e^{-t^\alpha L_\alpha}\r)^{2M_1+1}g\in T_{2,v}^{p,\infty}(\mathbb{R}_{+}^{n+1}).$$
Moreover, for any $f\in H_{L_\alpha}^p(\mathbb{R}^n)\cap L^2(\mathbb{R}^n)$, from the definition of
$T_2^{p,\infty}(\mathbb{R}_{+}^{n+1})$ and Lemma \ref{l2.14}, we deduce that
\begin{align*}
&\int_{\mathbb{R}}f(x)g(x)dx\\
&\hs=C\int_{\mathbb{R}_{+}^{n+1}}t^\alpha L_\alpha e^{-t^\alpha L_\alpha}f(x)\lf(t^\alpha L_\alpha\r)^M e^{-t^\alpha L_\alpha}
\lf(I-e^{-t^\alpha L_\alpha}\r)^{2M_1+1}g(x)\frac{dx dt}{t}\\
&\hs=C\int_{\mathbb{R}_{+}^{n+1}}\lf(t^\alpha L_\alpha\r)^{M_0} e^{-t^\alpha L_\alpha}f(x)\lf(t^\alpha L_\alpha\r)^{M-M_0+1}
e^{-t^\alpha L_\alpha}\lf(I-e^{-t^\alpha L_\alpha}\r)^{2M_1+1}g(x)\frac{dx dt}{t}.
\end{align*}
Let $O_k:=\{(x,t):|x|<k,\frac{1}{k}<t<k\}$,
$$h:=\lf(t^\alpha L_\alpha\r)^{M-M_0+1} e^{-t^\alpha L_\alpha}\lf(I-e^{-t^\alpha L_\alpha}\r)^{2M_1+1}g,$$
and $h_k:=\chi_{O_k}h$ for each $k\in\mathbb{N}$. Then $h_k\in T_{2,c}^{p,\infty}$ and
$\|h_k-h\|_{T_{2}^{p,\infty}}\rightarrow 0 $, as $k\rightarrow \infty$.
It then follows from Lemma \ref{l4.4} that $\pi_{L_\alpha,M_0}h_k\in \mathrm{VMO}_{L\alpha,M_0}
^{\frac{1}{p}-1}(\mathbb{R}^n)\cap L^2(\mathbb{R}^n)$ and
\begin{equation}\label{4.3}
\|\pi_{L_\alpha,M_0}(h-h_k)\|_{\mathrm{BMO}_{L\alpha,M_0}^{\frac{1}{p}-1}(\mathbb{R}^n)}\ls
\|h-h_k\|_{T_{2}^{p,\infty}}\rightarrow 0,
\end{equation}
as $k\rightarrow \infty$. Then, by the dominated convergence theorem, we further have
\begin{align*}
\int_{\mathbb{R}^n}f(x)g(x)dx&=C\int_{\mathbb{R}_{+}^{n+1}}\lf(t^\alpha L_\alpha\r)^{M_0} e^{-t^\alpha L_\alpha}f(x)h(x,t)\frac{dxdt}{t}\\
&=C\lim_{k \rightarrow \infty}\int_{\mathbb{R}_{+}^{n+1}}\lf(t^\alpha L_\alpha\r)^{M_0} e^{-t^\alpha L_\alpha}f(x)h_k(x,t)\frac{dxdt}{t}\\
&=C\lim_{k \rightarrow \infty}\int_{\mathbb{R}^{n}}f(x)\int_{0}^\infty \lf(t^\alpha L_\alpha\r)^{M_0} e^{-t^\alpha L_\alpha}h_k(x,t)\frac{dt}{t}dx\\
&=C\lim_{k \rightarrow \infty}\langle f,\pi_{L_\alpha,M_0}(h_k)\rangle\\
&=C\langle f,\pi_{L_\alpha,M_0}(h)\rangle.
\end{align*}
Since $H_{L_\alpha}^p(\mathbb{R}^n)\cap L^2(\mathbb{R}^n)$ is dense in the space
$H_{L_\alpha}^p(\mathbb{R}^n)$, we then obtain that
\begin{align*}
g=C\pi_{L_\alpha,M_0}(h)
\end{align*}
in $\mathrm{VMO}_{L\alpha,M_0}^{\frac{1}{p}-1}(\rn)$. Let $g_k=C\pi_{L_\alpha,M_0}(h_k)$. Then $g_k\in
\mathrm{VMO}_{L\alpha,M_0}^{\frac{1}{p}-1}(\mathbb{R}^n)\cap L^2(\mathbb{R}^n)$. By \eqref{4.3}, we conclude that
$\|g-g_k\|_{\mathrm{BMO}_{L\alpha,M_0}^{\frac{1}{p}-1}(\mathbb{R}^n)}\rightarrow 0$, as $k\rightarrow \infty$,
which completes the proof of Lemma \ref{l4.5}.
\end{proof}

We are now ready to prove Theorem \ref{t4.1}.

\begin{proof}[Proof of Theorem \ref{t4.1}]
From the fact that $\mathrm{VMO}_{L_{\alpha}, M_0}^{\frac{1}{p}-1}\left(\mathbb{R}^{n}\r)\subset
\mathrm{BMO}_{L_{\alpha}, M_0}^{\frac{1}{p}-1}\left(\mathbb{R}^{n}\r)$, we deduce that
$$ \lf( \mathrm{BMO}_{L_{\alpha}, M_0}^{\frac{1}{p}-1}\lf(\mathbb{R}^{n}\r)\r)^*\subset
\lf( \mathrm{VMO}_{L_{\alpha}, M_0}^{\frac{1}{p}-1}\lf(\mathbb{R}^{n}\r)\r)^*.$$
This, together with Theorem \ref{t2.11}, shows that
$$H_{L_\alpha}^p(\mathbb{R}^{n}) \subset\left(\mathrm{VMO}_{L_{\alpha}, M_0}^{\frac{1}{p}-1}
(\mathbb{R}^{n})\right)^{*}.$$
	
On the other hand, let $\ell \in(\mathrm{VMO}_{L_{\alpha}, M_0}^{\frac{1}{p}-1}(\mathbb{R}^{n}))^{*}$.
Let $M,M_1\in \mathbb{N}$ with $M_{1} \geq M_0$ and $ M \geq 2 M_{1}+M_0-1$. For any
$f \in \mathrm{VMO}_{L_{\alpha}, M_0}^{\frac{1}{p}-1}(\mathbb{R}^{n}) \cap L^{2}(\mathbb{R}^{n})$, by
Theorem \ref{t3.4}, we find that
$$\lf(t^\alpha L_\alpha\r)^{M-M_0+1} e^{-t^\alpha L_\alpha}\lf(I-e^{-t^\alpha L_\alpha}\r)^{2M_1+1}
f \in T_{2,v}^{p,\infty}(\mathbb{R}_{+}^{n+1}).$$
Further, from Lemma \ref{l4.4}, it follows that
$$\ell \circ \pi_{L_\alpha,M_0} \in \left(T_{2,v}^{p,\infty}(\mathbb{R}_{+}^{n+1})\right)^{*}.$$
This, together with Theorem \ref{t3.3}, tells us that there exists $g \in \widetilde{{T}_{2}^p}
(\mathbb{R}_{+}^{n+1})$ such that
\begin{align*}
\ell(f) &=\ell \circ \pi_{L_\alpha,M_0}\left(\lf(t^\alpha L_\alpha\r)^{M-M_0+1} e^{-t^\alpha L_\alpha}
\lf(I-e^{-t^\alpha L_\alpha}\r)^{2M_1+1} f \right) \\
&=\int_{\mathbb{R}_{+}^{n+1}}\lf(t^\alpha L_\alpha\r)^{M-M_0+1} e^{-t^\alpha L_\alpha}\lf(I-e^{-t^\alpha
L_\alpha}\r)^{2M_1+1} f(x) g(x, t) \frac{d x d t}{t} .
\end{align*}
Since $g \in \widetilde{{T}_{2}^p}(\mathbb{R}_{+}^{n+1})$, it follows that there exist $T_2^p$-atoms
$\left\{a_{j}\right\}_{j}$ and $\{\lambda\}_{j} \subset \mathbb{C}$ such that $g=$ $\sum_{j} \lambda_{j} a_{j}$
and $\sum_{j}\left|\lambda_{j}\right|^p \le C_p\|g\|_{\widetilde{{T}_2^p}}^p$. For each $k \in \mathbb{N}$,
let $g_{k} \equiv \sum_{j=1}^{k} \lambda_{j} a_{j}$. Then $g_{k} \in T_{2}^{2}(\mathbb{R}_{+}^{n+1})\cap
T_2^p(\mathbb{R}_{+}^{n+1})$. By Lemma \ref{l4.3}, we conclude that
$$\pi_{L_\alpha,{M-M_0+1},2M_1}(g) \in H_{L_\alpha}^p(\mathbb{R}^{n}), \pi_{L_\alpha,{M-M_0+1},2M_1}(g_k)
\in H_{L_\alpha}^p(\mathbb{R}^{n})$$
and
$$
\left\|\pi_{L_\alpha,{M-M_0+1},2M_1}(g) -\pi_{L_\alpha,{M-M_0+1},2M_1}(g_k)
\right\|_{H_{L_\alpha}^p(\mathbb{R}^{n})} \lesssim\left\|g-g_{k}\right\|_{\widetilde{{T}_2^p}} \rightarrow 0,
$$
as $k \rightarrow \infty$. Using the dominated convergence theorem, we further obtain that
$$
\begin{aligned}
\ell(f) &= \int_{\mathbb{R}_{+}^{n+1}}\lf(t^\alpha L_\alpha\r)^{M-M_0+1} e^{-t^\alpha L_\alpha}
\lf(I-e^{-t^\alpha L_\alpha}\r)^{2M_1+1} f(x) g(x, t) \frac{d x d t}{t}\\
& \sim \lim _{k \rightarrow \infty} \int_{\mathbb{R}_{+}^{n+1}} \lf(t^\alpha L_\alpha\r)^{M-M_0+1}
e^{-t^\alpha L_\alpha}\lf(I-e^{-t^\alpha L_\alpha}\r)^{2M_1+1} f(x) g_k(x, t) \frac{d x d t}{t} \\
& \sim \lim _{k \rightarrow \infty} \int_{\mathbb{R}^{n}} f(x)\int_{0}^\infty  \lf(t^\alpha
L_\alpha\r)^{M-M_0+1} e^{-t^\alpha L_\alpha}\lf(I-e^{-t^\alpha L_\alpha}\r)^{2M_1+1}  g_k(x, t) \frac{dt}{t} d x \\
& \sim \lim _{k \rightarrow \infty} \int_{\mathbb{R}^{n}} f(x) \pi_{L_\alpha, {M-M_0+1},
2M_{1}}\left(g_{k}\right)(x) d x \sim \lim _{k \rightarrow \infty}\left\langle f, \pi_{L_\alpha,
{M-M_0+1}, 2M_{1}}\left(g_{k}\right)\right\rangle \\
& \sim \sum_{j} \lambda_{j}\left\langle f, \pi_{L_\alpha, {M-M_0+1}, 2M_{1}}\left(a_{j}
\right)\right\rangle \sim\left\langle f, \pi_{L_\alpha, {M-M_0+1}, 2M_{1}}(g)\right\rangle .
\end{aligned}
$$
	
Applying Lemma \ref{l4.5} and a density argument, we finish the proof of Theorem \ref{t4.1}.
\end{proof}
	
\subsection*{Acknowledgments}
\addcontentsline{toc}{section}{Acknowledgments} 
\hskip\parindent
	The second author was supported by Chinese Universities Scientific Fund (Grant No. 2023TC107). The third author was supported by the National Natural Science Foundation of China (Grant Nos. 11871254 and 12071431), the Fundamental Research Funds for the Central Universities (Grant No. lzujbky-2021-ey18) and the Innovative Groups of Basic Research in Gansu Province (Grant No. 22JR5RA391).
	
\subsection*{Data availability}
\addcontentsline{toc}{section}{Data availability} 
\hskip\parindent
This manuscript has no associated data.

\section*{Declarations}
\subsection*{Conflict of interest}
\addcontentsline{toc}{section}{Conflict of interest} 
\hskip\parindent
The authors declare that they have no conflict of interest.

\bigskip
	
\medskip
	
\noindent Qiumeng Li and Haibo Lin (Corresponding author)
	
\smallskip
	
\noindent College of Science, China Agricultural University,
Beijing 100083, People's Republic of China
	
\smallskip
	
\noindent{\it E-mails}: \texttt{liqiumeng@cau.edu.cn} (Q. Li)
	
\hspace{1.08cm}\texttt{haibolincau@126.com} (H. Lin)
	
\bigskip
	
\noindent Sibei Yang
	
\smallskip
	
\noindent School of Mathematics and Statistics, Lanzhou University, Lanzhou 730000, People's Republic of China
	
\smallskip
	
\noindent{\it E-mails}: \texttt{yangsb@lzu.edu.cn} (S. Yang)


\begin{thebibliography}{10}
\bibitem{adm05}P. Auscher, X. T. Duong and A. McIntosh, Boundedness of Banach space valued singular integral operators 
and Hardy spaces, Unpublished manuscript (2005).
		
\vspace{-0.28cm}
		
\bibitem{amr08}P. Auscher, A. McIntosh and E. Russ, Hardy spaces of differential forms on Riemannian manifolds,
J. Geom. Anal. 18 (2008), 192-248.
		
\vspace{-0.28cm}
		
\bibitem{ar03}P. Auscher and E. Russ, Hardy spaces and divergence operators on strongly Lipschitz domains 
of $\mathbb{R}^n$, J. Funct. Anal. 201 (2003), 148-184.
		
\vspace{-0.28cm}
		
\bibitem{bgjp19}K. Bogdan, T. Grzywny, T. Jakubowski and D. Pilarczyk, Fractional Laplacian 
with Hardy potential, Comm. Partial Differential Equations 44 (2019), 20-50.
		
\vspace{-0.28cm}
		
\bibitem{blm04}K. Broderix, H. Leschke and P. M\"uller, Continuous integral kernels for unbounded Schr\"odinger 
semigroups and their spectral projections, J. Funct. Anal. 212 (2004), 287-323.

\vspace{-0.28cm}
		
\bibitem{bb21}T. A. Bui and T. Q. Bui, Maximal regularity of parabolic equations associated to generalized Hardy 
operators in weighted mixed-norm spaces, J. Differential Equations 303 (2021), 547-574.

\vspace{-0.28cm}
		
\bibitem{bd21}T. A. Bui and P. D'Ancona, Generalized Hardy operators, Nonlinearity 36 (2023), 171-198.
		
\vspace{-0.28cm}
		
\bibitem{bn22}T. A. Bui and G. Nader, Hardy spaces associated to generalized Hardy operators and applications, 
NoDEA Nonlinear Differential Equations Appl. 29 (2022), Paper No. 40, 40 pp.
		
\vspace{-0.28cm}

\bibitem{cy22}
M. Cao and K. Yabuta, VMO spaces associated with Neumann Laplacian. J. Geom. Anal. 32(2022), no. 2, Paper No. 59, 47 pp.

\vspace{-0.28cm}
	
\bibitem{cms85}R. R. Coifman, Y. Meyer and E. M. Stein, Some new function spaces and their applications to 
harmonic analysis, J. Funct. Anal. 62 (1985), 304-335.

\vspace{-0.28cm}

\bibitem{cw71}
R. R. Coifman and G. Weiss, Analyse harmonique non-commutative sur certains espaces homog$\grave{e}$nes. (French) \'Etude de certaines int\'egrales singuli$\grave{e}$res. Lecture Notes in Mathematics, Vol. 242. Springer-Verlag, Berlin-New York, 1971. {\rm v}+160 pp.
	
\vspace{-0.28cm}
		
\bibitem{cw77}R. R. Coifman and G. Weiss, Extensions of Hardy spaces and their use in analysis, 
Bull. Amer. Math. Soc. 83 (1977), 569-645.

\vspace{-0.28cm}
		
\bibitem{d95}E. B. Davies, Spectral Theory and Differential Operators, Cambridge Studies in Advanced Mathematics
42, Cambridge University Press, Cambridge, 1995. 
		
\vspace{-0.28cm}
		
\bibitem{ddsty08}D. Deng, X. T. Duong, L. Song, C. Tan and L. Yan, Functions of vanishing mean oscillation associated 
with operators and applications, Michigan Math. J. 56 (2008), 529-550.
		
\vspace{-0.28cm}
		
\bibitem{ddy05}D. Deng, X. T. Duong and L. Yan,  A characterization of the Morrey--Campanato spaces, 
Math. Z. 250 (2005), 641-655.
		
\vspace{-0.28cm}
			
\bibitem{dl13}X. T. Duong and J. Li, Hardy spaces associated to operators satisfying Davies--Gaffney estimates and 
bounded holomorphic functional calculus, J. Funct. Anal. 264 (2013), 1409-1437.
		
\vspace{-0.28cm}
		
\bibitem{dxy07}X. T. Duong, J. Xiao and L. Yan, Old and new Morrey spaces with heat kernel bounds, J. Fourier Anal. 
Appl. 13 (2007), 87-111.
		
\vspace{-0.28cm}
		
\bibitem{dy051}X. T. Duong and L. Yan, Duality of Hardy and BMO spaces associated with operators with heat kernel bounds, 
J. Amer. Math. Soc. 18 (2005), 943-973.
		
\vspace{-0.28cm}
		
\bibitem{dy052}X. T. Duong and L. Yan, New function spaces of BMO type, the John--Nirenberg inequality, interpolation, 
and applications, Comm. Pure Appl. Math. 58 (2005), 1375-1420.

\vspace{-0.28cm}
	
\bibitem{dgmtz05}J. Dziuba\'nski, G. Garrig\'os, T. Mart\'inez, J. L. Torrea and J. Zienkiewicz, BMO spaces related to
Schr\"odinger operators with potentials satisfying a reverse H\"older inequality, Math. Z. 249 (2005), 
329-356.

\vspace{-0.28cm}

\bibitem{fms21} R. L. Frank, K. Merz and H. Siedentop, Equivalence of Sobolev norms involving generalized Hardy operators, 
Int. Math. Res. Not. IMRN 2021, 2284-2303.
		
\vspace{-0.28cm}

\bibitem{g141}L. Grafakos, Classical Fourier Analysis, Third edition, Graduate Texts in Mathematics 249,
Springer, New York, 2014. 
		
\vspace{-0.28cm}
		
\bibitem{g142}L. Grafakos, Modern Fourier Analysis, Third edition, Graduate Texts in Mathematics 250,
Springer, New York, 2014. 
		
\vspace{-0.28cm}
		
\bibitem{h77} I. W. Herbst, Spectral theory of the operator $(p\sp{2}+m\sp{2})\sp{1/2}-Ze\sp{2}/r$, 
Comm. Math. Phys. 53 (1977), 285-294.
				
\vspace{-0.28cm}
		
\bibitem{hlmmy11}S. Hofmann, G. Lu, D. Mitrea, M. Mitrea and L. Yan, Hardy spaces associated to non-negative 
self-adjoint operators satisfying Davies--Gaffney estimates, Mem. Amer. Math. Soc. 214 (2011), no. 1007, vi+78 pp.
		
\vspace{-0.28cm}
		
\bibitem{hm09}S. Hofmann and S. Mayboroda, Hardy and BMO spaces associated to divergence form elliptic operators, 
Math. Ann. 344 (2009), 37-116.

\vspace{-0.28cm}

\bibitem{jxhs10}M. Jeng, S.-L.-Y.  Xu, E. Hawkins and J. M. Schwarz, On the nonlocality of the fractional Schr\"odinger equation, 
J. Math. Phys. 51 (2010), no. 6, 062102, 6 pp.

\vspace{-0.28cm}
	
\bibitem{jy10}R. Jiang and D. Yang, Generalized vanishing mean oscillation spaces associated with divergence form 
elliptic operators, Integral Equations Operator Theory 67 (2010), 123-149.
		
\vspace{-0.28cm}
		
\bibitem{jy111}R. Jiang and D. Yang, Orlicz--Hardy spaces associated with operators satisfying Davies--Gaffney estimates, 
Commun. Contemp. Math. 13 (2011), 331-373.
		
\vspace{-0.28cm}
		
\bibitem{jy112}R. Jiang and D. Yang, Predual spaces of Banach completions of Orlicz--Hardy spaces associated with operators, 
J. Fourier Anal. Appl. 17 (2011), 1-35.
		
\vspace{-0.28cm}

\bibitem{jn61}F. John and L. Nirenberg, On functions of bounded mean oscillation, Comm. Pure Appl. Math. 14 (1961),
415-426.

\vspace{-0.28cm}
		
\bibitem{rr06}V. S. Rabinovich and S. Roch, The essential spectrum of Schr\"odinger operators on lattices, J. Phys. A 
39 (2006), 8377-8394.
		
\vspace{-0.28cm}
		
\bibitem{s82}B. Simon, Schr\"odinger semigroups, Bull. Amer. Math. Soc. (N.S.) 7 (1982), 447-526.
		
\vspace{-0.28cm}
	
\bibitem{t86}A. Torchinsky, Real-variable methods in harmonic analysis. Reprint of the 1986 original [Dover, New York; MR0869816]. Dover Publications, Inc., Mineola, NY, 2004. xiv+462 pp. ISBN: 0-486-43508-3

\vspace{-0.28cm}

\bibitem{ttd21}N. N. Trong, L. X. Truong and T. D. Do, Boundedness of second-order Riesz transforms on weighted Hardy 
and $\mathrm{BMO}$ spaces associated with Schr\"odinger operators, C. R. Math. Acad. Sci. Paris 359 (2021), 687-717.
		
\vspace{-0.28cm}
		
\bibitem{u78}A. Uchiyama, On the compactness of operators of Hankel type, Tohoku Math. J. (2) 30 (1978), 163-171.
		
\vspace{-0.28cm}
		
\bibitem{y04}L. Yan, Littlewood--Paley functions associated to second order elliptic operators, Math. Z. 
246 (2004), 655-666.
		
\vspace{-0.28cm}
		
\bibitem{y08}L. Yan, Classes of Hardy spaces associated with operators, duality theorem and applications, Trans. Amer. 
Math. Soc. 360 (2008), 4383-4408.
		
\vspace{-0.28cm}
		
		
\bibitem{yyz10}Da. Yang, Do. Yang and Y. Zhou, Localized BMO and BLO spaces on RD-spaces and applications to Schr\"odinger 
operators, Commun. Pure Appl. Anal. 9 (2010), 779-812.
		
\vspace{-0.28cm}
		
\bibitem{zy09}Q. Zheng and X. Yao, Higher-order Kato class potentials for Schr\"odinger operators, Bull. Lond. Math. Soc. 
41 (2009), 293-301.
		
\end{thebibliography}
\end{document}